\documentclass[a4paper,11pt]{amsart}

\usepackage{amsfonts,amssymb,amscd}

\usepackage{textcomp}

\usepackage{graphicx,tikz}

\DeclareMathOperator{\Stab}{Stab}

\DeclareMathOperator{\Aut}{Aut}

\DeclareMathOperator{\rk}{rk}

\theoremstyle{plain}

\newtheorem{thm}{Theorem}[section]
\newtheorem*{thmA}{Theorem A}
\newtheorem*{thmB}{Theorem B}

\newtheorem{lem}[thm]{Lemma}
\newtheorem{pro}[thm]{Proposition}
\newtheorem{cor}[thm]{Corollary}

\theoremstyle{definition}

\newtheorem{dfn}[thm]{Definition}

\newcommand{\F}{\mathbb{F}}
\newcommand{\Z}{\mathbb{Z}}

\newcommand{\T}{\mathcal{T}}
\newcommand{\B}{\mathcal{B}}

\DeclareMathOperator{\img}{im}

\begin{document}

\title[GGS-groups: congruence quotients and Hausdorff dimension]%
{GGS-groups: order of congruence quotients and Hausdorff dimension}

\author{Gustavo A. Fern\'andez-Alcober}
\address{Matematika Saila\\ Euskal Herriko Unibertsitatea
\\ 48080 Bilbao (Spain)}
\email{gustavo.fernandez@ehu.es}

\author{Amaia Zugadi-Reizabal}
\address{Matematika Saila\\ Euskal Herriko Unibertsitatea
\\ 48080 Bilbao (Spain)}
\email{amaia.zugadi@ehu.es}

\thanks{Supported by the Spanish Government, grant
MTM2008-06680-C02-02, partly with FEDER funds, and by the
Basque Government, grant IT-460-10.
The second author is also supported by grant BFI07.95 of the
Basque Government.}

\begin{abstract}
If $G$ is a GGS-group defined over a $p$-adic tree, where $p$ is an
odd prime, we calculate the order of the congruence quotients $G_n=G/\Stab_G(n)$
for every $n$.
If $G$ is defined by the vector $\mathbf{e}=(e_1,\ldots,e_{p-1})\in\F_p^{p-1}$,
the determination of the order of $G_n$ is split into three cases,
according as $\mathbf{e}$ is non-symmetric, non-constant symmetric,
or constant.
The formulas that we obtain only depend on $p$, $n$, and the rank of the
circulant matrix whose first row is $\mathbf{e}$.
As a consequence of these formulas, we also obtain the Hausdorff dimension
of the closures of all GGS-groups over the $p$-adic tree.
\end{abstract}

\maketitle

\section{Introduction}

Subgroups of the group of automorphisms of a regular rooted tree have turned
out to be a source of many interesting examples in group theory.
Particular attention has been given to the so-called Grigorchuk groups
and to the Gupta-Sidki group, introduced in \cite{gri} and \cite{gupta-sidki},
respectively.
The second of the Grigorchuk groups and the Gupta-Sidki group are
particular instances of the family of \emph{GGS-groups} (GGS after Grigorchuk,
Gupta, and Sidki, a term coined by Gilbert Baumslag), to which this paper is devoted.
We work over the $p$-adic tree, where $p$ is an odd prime, and we determine
the order of all congruence quotients of GGS-groups; these are the automorphism
groups induced by GGS-groups on the finite trees which are obtained by truncating
the $p$-adic tree at every level.
As a consequence, we also obtain the Hausdorff dimension of the closures of
GGS-groups.

\vspace{10pt}

Before defining GGS-groups and stating our main results, it is convenient
to recall some concepts from the theory of automorphisms of rooted trees.
If $m\ge 2$ is an integer and $X=\{1,\ldots,m\}$, the $m$-adic tree $\T$
is the tree whose set of vertices is the free monoid $X^*$, where a word
$u$ is a descendant of $v$ if $u=vx$ for some $x\in X$.
If we consider only words of length $\le n$, then we have a finite tree
$\T_n$, which we refer to as \emph{the tree $\T$ truncated at level $n$\/}.
The group $\Aut\T$ of all automorphisms of $\T$ is a profinite group
with respect to the topology induced by the filtration of the level
stabilizers $\Stab(n)$, and we have $\Aut\T\cong \varprojlim_n \Aut\T_n$.
The stabilizer $\Stab(n)$ of the $n$th level of $\T$ is the normal subgroup
of $\Aut\T$ consisting of all automorphisms leaving fixed all words of
length $n$ (and, consequently, also all vertices of $\T_n$).
These stabilizers can be considered as natural congruence subgroups for $\Aut\T$.
If $G$ is a subgroup of $\Aut\T$ and we put $\Stab_G(n)=\Stab(n)\cap G$,
then we refer to the quotient $G_n=G/\Stab_G(n)$ as the
\emph{$n$th congruence quotient\/} of $G$.
Since the kernel of the action of $G$ on $\T_n$ is $\Stab_G(n)$, it
follows that $G_n$ can be naturally seen as a subgroup of $\Aut\T_n$.

\vspace{10pt}

If an automorphism $g$ fixes a vertex $u$, then the restriction of $g$ to the
subtree hanging from $u$ induces an automorphism $g_u$ of $\T$.
In particular, if $g\in\Stab(1)$ then $g_i$ is defined for every $i=1,\ldots,m$,
and we can consider the map
\[
\begin{matrix}
\psi & : & \Stab(1) & \longrightarrow & \Aut\T \times \overset{m}{\cdots}
\times \Aut\T
\\
& & g & \longmapsto & (g_1,\ldots,g_m).
\end{matrix}
\]
Clearly, $\psi$ is a group isomorphism.

On the other hand, any $g\in\Aut\T$ can be completely determined by describing how $g$
sends the descendants of every vertex $u$ to the descendants of $g(u)$.
This can be done by indicating, for every $x\in X$, the element $\alpha(x)\in X$
such that $g(ux)=g(u)\alpha(x)$.
Then $\alpha$ is a permutation of $X$, which we call the \emph{label} of
$g$ at $u$, and we denote by $g_{(u)}$.
The set of all labels of $g$ constitutes the \emph{portrait} of $g$.
Thus $g$ is determined by its portrait.
We have the following rules for labels under composition and inversion:
\begin{equation}
\label{label laws}
(fg)_{(u)}
=
f_{(u)}g_{(f(u))}
\quad
\text{and}
\quad
(f^{-1})_{(u)}
=
(f_{(f^{-1}(u))})^{-1}.
\end{equation}

\vspace{10pt}

An important automorphism of $\T$ is the automorphism that permutes the
$m$ subtrees hanging from the root rigidly according to the permutation
$(1\ 2\ \ldots \ m)$.
This is called a \emph{rooted automorphism\/} and will be denoted by the
letter $a$.
Since $a$ has order $m$, it makes sense to write $a^k$ for $k\in\Z/m\Z$.
Now, given a non-zero vector $\mathbf{e}=(e_1,\ldots,e_{m-1})\in (\Z/m\Z)^{m-1}$,
we can define recursively an automorphism $b$ of $\T$ via
\[
\psi(b) = (a^{e_1},\ldots,a^{e_{m-1}},b).
\]
We say that the subgroup $G=\langle a,b \rangle$ of $\Aut\T$ is the
\emph{GGS-group\/} corresponding to the \emph{defining vector\/} $\mathbf{e}$.
If $m=2$ then there is only one GGS-group, which is isomorphic to $D_{\infty}$,
the infinite dihedral group.
The second Grigorchuk group is obtained by choosing $m=4$ and $\mathbf{e}=(1,0,1)$,
and the Gupta-Sidki group arises for $m$ equal to an odd prime and
$\mathbf{e}=(1,-1,0,\ldots,0)$.
The groups corresponding to $\mathbf{e}=(1,0,\ldots,0)$ and arbitrary $m$
have also deserved special attention.
In the case $m=3$, this group was introduced by Fabrykowski and Gupta in \cite{fab-gupta}.
As a reference for GGS-groups, the reader can consult Section 2.3 of the monograph
\cite{branch} by Bartholdi, Grigorchuk, and \v{S}uni\'{c}, the habilitation thesis
\cite{rozh} of Rozhkov, or the papers \cite{vov} by Vovkivsky and \cite{per,per2}
by Pervova.

\vspace{10pt}

Little is known about the orders of the congruence quotients $G_n$ when $G$
is a GGS-group.
In the case that $\mathbf{e}=(1,0,\ldots,0)$ and $m=p$ is a prime,
\v{S}uni\'{c} found in \cite{sunic} that, for every $n\ge 2$,
\[
\log_p |G_n|
=
\begin{cases}
p^{n-1}+1,
&
\text{if $p$ is odd,}
\\
2^{n-2}+2,
&
\text{if $p=2$.}
\end{cases}
\]
Hence we may always assume that $m\ge 3$, as far as the problem of determining
$|G_n|$ is concerned.
To the best of our knowledge, the only other cases in which the order of $G_n$
has been determined for every $n$ correspond to $m=3$.
For the Gupta-Sidki group, Sidki himself (see \cite{sid}) proved that
\[
\log_3 |G_n| = 2 \cdot 3^{n-2} + 1,
\quad
\text{for every $n\ge 2$.}
\]
On the other hand, for $\mathbf{e}=(1,1)$, Bartholdi and Grigorchuk showed in
\cite{on parabolic} that
\[
\log_3 |G_n| = \frac{3^n+2n+3}{4},
\quad
\text{for every $n\ge 2$.}
\]

\vspace{10pt}

From now onwards, we assume that $m$ is equal to an odd prime $p$, and so
$\T$ stands for the $p$-adic tree.
The first of our main results is the determination of the order of $G_n$
for \textit{all\/} GGS-groups under this assumption.
Before giving the statement of the theorem, we introduce some notation.
Given a vector $\mathbf{a}=(a_1,\ldots,a_n)$, we write
$C(\mathbf{a})$ to denote the circulant matrix generated by $\mathbf{a}$,
i.e.\ the matrix of size $n\times n$ whose first row is $\mathbf{a}$, and
every other row is obtained from the previous one by applying a shift of
length one to the right.
In other words, the entries of $C(\mathbf{a})$ are $c_{ij}=a_{j-i+1}$,
where $a_k$ is defined for every integer $k$ by reducing $k$ modulo $n$
to a number between $1$ and $n$.
If $\mathbf{e}$ is the defining vector of a GGS-group, then we write
$C(\mathbf{e},0)$ for the circulant matrix $C(e_1,\ldots,e_{p-1},0)$
over $\F_p$.
We say that $\mathbf{e}$ is \emph{symmetric\/} if $e_i=e_{p-i}$ for all
$i=1,\ldots,p-1$.

\begin{thmA}
Let $G$ be a GGS-group over the $p$-adic tree, where $p$ is an odd prime,
and let $\mathbf{e}$ be the defining vector of $G$.
Then, for every $n\ge 2$, we have
\[
\log_p |G_n|
=
tp^{n-2}+1-\delta\,\frac{p^{n-2}-1}{p-1}
-\varepsilon\,\frac{p^{n-2}-(n-2)p+n-3}{(p-1)^2},
\]
where $t$ is the rank of the circulant matrix $C(\mathbf{e},0)$,
\begin{equation*}
\delta
=
\begin{cases}
1, & \text{if $\mathbf{e}$ is symmetric,}
\\
0, & \text{otherwise,}
\end{cases}
\qquad
\text{and}
\qquad
\varepsilon
=
\begin{cases}
1, & \text{if $\mathbf{e}$ is constant,}
\\
0, & \text{otherwise.}
\end{cases}
\end{equation*}
\end{thmA}

\vspace{10pt}

If $\sigma=(1\ 2\ \ldots \ p)$, then the automorphisms whose portrait
consists only of powers of $\sigma$ form a Sylow pro-$p$ subgroup of $\Aut\T$,
which we denote by $\Gamma$.
Observe that, under the assumption $m=p$ that we have made, all GGS-groups
are subgroups of $\Gamma$.
According to Theorem 1 of \cite{vov}, the requirement that $\mathbf{e}$
is non-zero implies that GGS-groups are infinite if $m=p$.
Since they are countable groups, they cannot be closed in the pro-$p$ group
$\Gamma$.
Our second main result is related to the Hausdorff dimension of the closures
of GGS-groups.

\vspace{10pt}

The determination of the Hausdorff dimension of closed subgroups of $\Gamma$
has received special attention in the last few years
(see \cite{abert-virag,fer-zug,siegh,sunic}).
The most natural choice is to calculate the Hausdorff dimension with respect
to the metric induced by the filtration of $\Gamma$ given by the level
stabilizers $\Stab_{\Gamma}(n)$.
In this case, it follows from a result of Abercrombie \cite{abercrombie}, and
Barnea and Shalev \cite{bar-sha}, that the Hausdorff dimension of the closure
$\overline G$ of a subgroup $G$ of $\Gamma$ is given by the following formula:
\begin{equation}
\label{dim-formula}
\dim_{\Gamma}\, \overline G
=
\liminf_{n\to\infty} \, \frac{\log_p |G_n|}{\log_p |\Gamma_n|}
=
(p-1) \liminf_{n\to\infty} \, \frac{\log_p |G_n|}{p^n}.
\end{equation}
As an immediate consequence of Theorem A, we get the Hausdorff dimension of
the closure of any GGS-group.

\begin{thmB}
Let $G$ be a GGS-group over the $p$-adic tree, where $p$ is an odd prime,
and let $\mathbf{e}$ be the defining vector of $G$.
Then
\[
\dim_{\Gamma} \overline G
=
\frac{(p-1)t}{p^2} - \frac{\delta}{p^2} - \frac{\varepsilon}{(p-1)p^2},
\]
where $t$ is the rank of the circulant matrix $C(\mathbf{e},0)$,
\begin{equation*}
\delta
=
\begin{cases}
1, & \text{if $\mathbf{e}$ is symmetric,}
\\
0, & \text{otherwise,}
\end{cases}
\qquad
\text{and}
\qquad
\varepsilon
=
\begin{cases}
1, & \text{if $\mathbf{e}$ is constant,}
\\
0, & \text{otherwise.}
\end{cases}
\end{equation*}
\end{thmB}

\vspace{10pt}

Our proof of Theorem A relies on finding some kind of branch structure inside a
GGS-group $G$.
In particular, if $\mathbf{e}$ is not constant, we show that $G$ is regular branch
(see Section \ref{sec: non-constant} for the definition).
This result had been previously proved by Pervova and Rozhkov for \emph{periodic}
GGS-groups.
On the other hand, it is worth mentioning that the theory of $p$-groups of maximal
class plays also a crucial role in the proof of Theorem A, particularly in the case
that $\mathbf{e}$ is constant.

\vspace{10pt}

\textit{Notation.}
The $i$th row and $j$th column of a matrix $C$ are denoted by $C_i$ and $C^j$, respectively.

\section{General properties of GGS-groups}

Throughout the paper, $a$ and $b$ denote the canonical generators of a
GGS-group $G$, and $b_i=b^{a^i}$ for every integer $i$.
Note that $b_i=b_j$ if $i\equiv j\pmod p$.
The images of the elements $b_i$ under the map $\psi$ of the introduction can be
easily described:
\begin{equation}
\label{tuplak}
\begin{split}
\psi(b_0) &= (a^{e_1},a^{e_2},\ldots,a^{e_{p-1}},b),
\\
\psi(b_1) &= (b,a^{e_1},\ldots,a^{e_{p-2}},a^{e_{p-1}}),
\\
&\vdots
\\
\psi(b_{p-1}) &= (a^{e_2},a^{e_3},\ldots,b,a^{e_1}).
\end{split}
\end{equation}

We begin with some easy facts about GGS-groups.

\begin{thm}
\label{Stab(1)}
If $G=\langle a,b \rangle$ is a GGS-group, then:
\begin{enumerate}
\item
$\Stab_G(1)=\langle b \rangle^G=\langle b_0,\ldots,b_{p-1}\rangle$ and
$G=\langle a \rangle \ltimes \Stab_G(1)$.
\item
$\Stab_G(2)\le G'\le\Stab_G(1)$.
\item
$|G:G'|=p^2$ and $|G:\gamma_3(G)|=p^3$.
\end{enumerate}
\end{thm}

\begin{proof}
One can easily check the equalities in part (i).
Thus $G/\Stab_G(1)$ is cyclic and $G'\le\Stab_G(1)$.

The quotient $G/G'=\langle aG',bG' \rangle$ is elementary abelian of
order at most $p^2$.
It follows that $G'/\gamma_3(G)=\langle [a,b]\gamma_3(G) \rangle$
has order at most $p$.
If $G'=\gamma_3(G)$ then $\gamma_i(G)=G'$ for every $i\ge 3$.
On the other hand, since $G$ is residually a finite $p$-group, the
intersection of all the $\gamma_i(G)$ is trivial.
Consequently $G'=1$, which is a contradiction, since $b^a\ne b$ by
(\ref{tuplak}).
We conclude that $|G':\gamma_3(G)|=p$.
Now, if $|G:G'|\le p$ then $G/G'$ is cyclic, and $G'=\gamma_3(G)$.
Hence we necessarily have $|G:G'|=p^2$, and (iii) follows.

It only remains to prove that $N=\Stab_G(2)$ is contained in $G'$.
Since $|G:G'|=p^2$, it suffices to prove that $|G/N:(G/N)'|=p^2$.
If $|G/N:(G/N)'|\le p$ then $G/N$, being a finite $p$-group, must be
cyclic.
This is a contradiction, since $\langle aN \rangle$ and $\langle bN \rangle$
are two different subgroups of order $p$ in $G/N$.
(Note that $\langle bN \rangle$ is contained in $\Stab_G(1)/N$ while
$\langle aN \rangle$ is not.)
\end{proof}

Now if $g\in\Stab_G(1)$, it readily follows from (\ref{tuplak}) and the previous
theorem that $g_i\in G$ for all $i=1,\ldots,p$.
Thus the image of $\Stab_G(1)$ under $\psi$ is actually contained in
$G\times \overset{p}{\cdots} \times G$, and so
\begin{equation}
\label{psi on kth stabilizer}
\psi(\Stab_G(k))\subseteq \Stab_G(k-1)\times \overset{p}{\cdots} \times \Stab_G(k-1)
\end{equation}
for all $k\ge 1$.
Another important property of the map $\psi$ is the following.

\begin{pro}
\label{surjective}
If $G$ is a GGS-group, then the composition of $\psi$ with the projection on any component
is surjective from $\Stab_G(1)$ onto $G$.
\end{pro}

\begin{proof}
Let us fix a position $i\in\{1,\ldots,p\}$, and let $j\in\{1,\ldots,p-1\}$ be such that
$e_j\ne 0$.
It follows from (\ref{tuplak}) that $\psi(b_{i-j})$ and $\psi(b_i)$ have the entries $a^{e_j}$
and $b$ in the $i$th component.
Since $G=\langle a,b \rangle=\langle a^{e_j},b \rangle$, the result follows.
\end{proof}

For every positive integer $n$, we can define an isomorphism $\psi_n$ from the stabilizer
of the first level in $\Aut\T_n$ to the direct product
$\Aut\T_{n-1}\times\overset{p}{\cdots}\times\Aut\T_{n-1}$, in the same way as $\psi$ is defined.
Since $G_n$ can be seen as a subgroup of $\Aut\T_n$, we can consider the restriction of
$\psi_n$ to $\Stab_{G_n}(1)$.
It follows from (\ref{psi on kth stabilizer}) that
\[
\psi_n(\Stab_{G_n}(k))\subseteq \Stab_{G_{n-1}}(k-1)\times \overset{p}{\cdots}
\times \Stab_{G_{n-1}}(k-1).
\]

Obviously, $G_1$ is of order $p$, generated by the image $\overline a$ of $a$.
Next we deal with $G_2$.
Let us write $\tilde g$ for the image of an element $g\in G$ in $G_2$.
Since $G_2=\langle \tilde a \rangle \ltimes \Stab_{G_2}(1)$, it suffices
to understand $\Stab_{G_2}(1)=\langle \tilde b_0,\ldots,\tilde b_{p-1} \rangle$.
Observe that $\psi_2$ sends $\Stab_{G_2}(1)$ into $G_1\times \overset{p}{\cdots} \times G_1$,
which can be identified with $\F_p^p$ under the linear map
\[
(\overline a^{i_1},\ldots,\overline a^{i_p}) \longmapsto (i_1,\ldots,i_p).
\]
This allows us to consider $\Stab_{G_2}(1)$ as a vector space over $\F_p$.

\vspace{5pt}

Before analyzing $G_2$ in the next theorem, we need the following lemma
(see Exercise 4 in Section 1 of the book \cite{berkovich}) about finite $p$-groups
of maximal class, which will be also used at some other places in the
paper.

\begin{lem}
\label{maximal class}
Let $P$ be a finite $p$-group such that $|P:P'|=p^2$.
If $P$ has an abelian maximal subgroup $A$, then $P$ is a group of maximal
class.
Furthermore, if $g_0\in P\smallsetminus A$, then:
\begin{enumerate}
\item
If $a\in A\smallsetminus \gamma_2(P)$, then $\gamma_2(P)/\gamma_3(P)$
is generated by the image of $[a,g_0]$.
\item
If $i\ge 2$ and $a\in \gamma_i(P)\smallsetminus \gamma_{i+1}(P)$,
then $\gamma_{i+1}(P)/\gamma_{i+2}(P)$ is generated by the image of $[a,g_0]$.
\end{enumerate}
\end{lem}

\begin{thm}
\label{order of G2}
Let $G$ be a GGS-group with defining vector $\mathbf{e}$, and put $C=C(\mathbf{e},0)$.
Then:
\begin{enumerate}
\item
The dimension of $\Stab_{G_2}(1)$ coincides with the rank $t$ of $C$.
\item
$G_2$ is a $p$-group of maximal class of order $p^{t+1}$.
\end{enumerate}
\end{thm}

\begin{proof}
(i)
If $\tilde g\in\Stab_{G_2}(1)$ and $\psi_2(\tilde g)=(\overline a^{i_1},\ldots,\overline a^{i_p})$,
where we consider the exponents $i_1,\ldots,i_p$ as elements of $\F_p$, we define
\[
\Psi_2(\tilde g) = (i_1,\ldots,i_p)\in \F_p^p.
\]
Observe that $\Psi_2$ is injective.

By (\ref{tuplak}),
\[
\Psi_2(\tilde b_0) = (e_1,e_2,\ldots,e_{p-1},0) = (\mathbf{e},0)
\]
coincides with the first row of $C$.
Since the components of the rest of the $b_i$ are obtained by permuting cyclically
those of $b_0$, and since $C=C(\mathbf{e},0)$, it follows that $\Psi_2(\tilde b_i)$
is the ($i+1$)st row of $C$.
Thus the dimension of $\Stab_{G_2}(1)$ coincides with the dimension of the subspace
of $\F_p^p$ generated by the rows of $C$, i.e.\ with the rank $t$ of the matrix $C$.

(ii)
We have
\[
|G_2| = |G_2:\Stab_{G_2}(1)||\Stab_{G_2}(1)| = p \cdot p^t = p^{t+1}.
\]
On the other hand, it follows from (ii) and (iii) of Theorem \ref{Stab(1)} that
$|G_2:G_2'|=p^2$.
Since $\Stab_{G_2}(1)$ is an abelian maximal subgroup of $G_2$, we conclude from
Lemma \ref{maximal class} that $G_2$ is a $p$-group of maximal class.
\end{proof}

As a consequence, we can improve part (ii) of Theorem \ref{Stab(1)}.

\begin{cor}
\label{Stab(2)}
If $G$ is a GGS-group, then $\Stab_G(2)\le \gamma_3(G)$.
\end{cor}

\begin{proof}
Since the defining vector $\mathbf{e}$ of $G$ is different from $(0,\ldots,0)$,
it is clear that the rank $t$ of the matrix $C(\mathbf{e},0)$ is at least $2$.
It follows from the previous theorem that $G_2=G/\Stab_G(2)$ is a $p$-group of maximal
class of order greater than or equal to $p^3$.
Thus $|G_2:\gamma_3(G_2)|=p^3=|G:\gamma_3(G)|$, and consequently
$\Stab_G(2)$ is contained in $\gamma_3(G)$.
\end{proof}

We have seen in Theorem \ref{Stab(1)} that $G'\le\Stab_G(1)$.
Next we want to characterize which elements of $\Stab_G(1)$ belong to $G'$.
This goal will be achieved in Theorem \ref{G'}.
If $g\in\Stab_G(1)=\langle b_0,\ldots,b_{p-1} \rangle$, then we can write
$g$ as a word in $b_0,\ldots,b_{p-1}$, i.e.\ we can write $g=\omega(b_0,\ldots,b_{p-1})$,
where $\omega=\omega(x_0,\ldots,x_{p-1})$ is a group word in the $p$ variables
$x_0,\ldots,x_{p-1}$.

\begin{dfn}
Let $\omega$ be a group word in the variables $x_0,\ldots,x_{p-1}$, where $p$
is a prime.
Then:
\begin{enumerate}
\item
The \emph{partial $p$-weight\/} of $\omega$ with respect to a variable $x_i$,
with $0\le i\le p-1$, is the sum of the exponents of $x_i$ in the expression
for $\omega$, considered as an element of $\F_p$.
\item
The \emph{total $p$-weight\/} of $\omega$ is the sum of all its partial $p$-weights.
\end{enumerate}
\end{dfn}

It is not difficult to give examples showing that the representation of an element
$g\in\Stab_G(1)$ as a word in $b_0,\ldots,b_{p-1}$ is not unique.
Our first step towards the proof of Theorem \ref{G'} will be to see that, however, the
partial and total $p$-weights are the same for all word representations.
For this purpose, we need the following lemma.

\begin{lem}
\label{rank circulant}
Let $p$ be a prime, and let $(a_0,\ldots,a_{p-1})\in\F_p^p$ be a non-zero vector.
If $C=C(a_0,\ldots,a_{p-1})$, then:
\begin{enumerate}
\item
$\rk C=p-m$, where $m$ is the multiplicity of $1$ as a root of the polynomial
$a(X)=a_0+a_1X+\cdots+a_{p-1}X^{p-1}$.
As a consequence, we have $\rk C<p$ if and only if $\sum_{i=0}^{p-1}\, a_i=0$.
\item
If $\mathbf{1}$ represents the column vector of length $p$ with all entries
equal to $1$, then
\[
\rk C
=
\rk \left(
C \mid \mathbf{1}
\right)
.
\]
\end{enumerate}
\end{lem}

\begin{proof}
If we consider the quotient ring $V=\F_p[X]/(X^p-1)$ as an $\F_p$-vector space, then
both
\[
\B=\{\overline 1,\overline X,\ldots,\overline{X^{p-1}}\}
\]
and
\[
\B'=\{\overline 1,\overline{X-1},\ldots,\overline{(X-1)^{p-1}}\}
\]
are bases of $V$.
Multiplication by $\overline{a(X)}$ defines a linear map $\varphi:V\rightarrow V$,
and the matrix of $\varphi$ with respect to $\B$ is $C$ (we construct the matrix by rows).
Thus $\rk C=\rk\varphi$.

On the other hand, we can write $a(X)=(X-1)^mb(X)$, with $b(X)\in\F_p[X]$ and $b(1)\ne 0$.
Let $b(X)=b_0+b_1(X-1)+\cdots+b_{k-1}(X-1)^{k-1}$, where $k=p-m$ and $b_0\ne 0$.
Then the matrix of $\varphi$ with respect to $\B'$ is the block matrix
\[
\begin{pmatrix} 0 & B\\ 0 & 0 \end{pmatrix},
\quad
\text{where}
\quad
B =
\begin{pmatrix}
b_0 & b_1 & \cdots & b_{k-2} & b_{k-1}
\\
0 & b_0 & \cdots & b_{k-3} & b_{k-2}
\\
\vdots & \vdots &  & \vdots & \vdots
\\
0 & 0 & \cdots & 0 & b_0
\end{pmatrix},
\]
since $\overline{(X-1)^i}=\overline 0$ in $V$ for all $i\ge p$.
Thus $\rk\varphi=k$, and (i) follows.

Let us now prove (ii).
We first prove that
\begin{equation}
\label{C+1 row}
\rk C = \rk \begin{pmatrix} C\\ 1 \ldots 1 \end{pmatrix}.
\end{equation}
Since $C$ is the matrix of $\varphi$ with respect to $\B$ constructed by rows,
it is clear that (\ref{C+1 row}) is equivalent to $\overline{1+X+\cdots+X^{p-1}}$
lying in the image of $\varphi$.
Note that, since we are working with coefficients in $\F_p$, we have
\[
1+X+\cdots+X^{p-1}=(X-1)^{p-1}.
\]
Since
\[
\varphi(\overline{(X-1)^{k-1}})
=
\overline{b_0(X-1)^{p-1}},
\]
and $b_0\ne 0$, it follows that $\overline{(X-1)^{p-1}}\in\img\varphi$,
as desired.

Now, since the transpose ${}^t C$ of $C$ is also a circulant matrix,
we can apply (\ref{C+1 row}) to ${}^t C$ and get
\[
\rk C = \rk {}^t C = \rk \begin{pmatrix} {}^t C \\ 1 \ldots 1 \end{pmatrix}
= \rk {}^t\!\left( C \mid \mathbf{1} \right)
= \rk \left( C \mid \mathbf{1} \right).
\]
\end{proof}

Let $g=\omega(b_0,\ldots,b_{p-1})$ be an arbitrary element of $\Stab_G(1)$, and
suppose that the partial $p$-weight of $\omega$ with respect to $x_i$ is $r_i$,
for $i=0,\ldots,p-1$.
It follows from (\ref{tuplak}) that
\begin{equation}
\label{components of word}
\psi(g) = (a^{m_1}\omega_1(b_0,\ldots,b_{p-1}),\ldots,a^{m_p}\omega_p(b_0,\ldots,b_{p-1})),
\end{equation}
where each $\omega_i$ is a word of total $p$-weight $r_i$ (and where $r_p$ is to be
understood as $r_0$), and
\begin{equation}
\label{exponent}
m_i = (r_0\ r_1\ \ldots \ r_{p-1}) C^i.
\end{equation}

\begin{thm}
\label{total weight}
Let $G$ be a GGS-group, and let $g\in\Stab_G(1)$.
Then the partial and total $p$-weights are the same for all representations of
$g$ as a word in $b_0,\ldots,b_{p-1}$.
\end{thm}

\begin{proof}
It suffices to see that, if $\omega$ is a word such that $\omega(b_0,\ldots,b_{p-1})=1$,
then the total $p$-weight of $\omega$ is $0$, and the partial $p$-weight $r_i$ of $\omega$
with respect to $x_i$ is equal to $0$, for every $i=0,\ldots,p-1$.
Obviously, the second assertion implies the first one, but the proof will go the other
way around.

As in (\ref{components of word}), we write
\begin{equation}
\label{word value}
\psi(\omega(b_0,\ldots,b_{p-1}))
=
(a^{m_1}\omega_1(b_0,\ldots,b_{p-1}),\ldots,a^{m_p}\omega_p(b_0,\ldots,b_{p-1})).
\end{equation}
Since this element is equal to $1$, it follows that $m_i=0$ for $i=1,\ldots,p$.
According to (\ref{exponent}), this means that
\[
(r_0\ r_1\ \ldots\ r_{p-1})C=(0\ 0\ \ldots\ 0).
\]
Now, since $\rk C=\rk (C\mid \mathbf{1})$ by Lemma \ref{rank circulant}, we also have
$(r_0\ r_1\ \ldots\ r_{p-1})\mathbf{1}=0$, that is,
\[
r_0+r_1+\cdots+r_{p-1}=0.
\]
This proves that the total $p$-weight of $\omega$ is $0$.

Now we return to (\ref{word value}).
Since $\omega(b_0,\ldots,b_{p-1})=1$ by hypothesis, then we also have
$\omega_i(b_0,\ldots,b_{p-1})=1$ for all $i=1,\ldots,p$.
Now, since the total $p$-weight of $\omega_i$ is $r_i$, it follows from the previous
paragraph that $r_i=0$.
\end{proof}

The independence of the partial and total $p$-weights from the word representation
allows us to give the following definition.

\begin{dfn}
Let $G$ be a GGS-group, and let $g\in\Stab_G(1)$.
We define the \emph{partial weight\/} of $g$ with respect to $b_i$, and the
\emph{total weight\/} of $g$, as the corresponding $p$-weights for any word
$\omega$ representing $g$.
\end{dfn}

We prefer to speak simply about weights instead of $p$-weights in the case of an element
$g\in\Stab_G(1)$, since all elements $b_i$ (with respect to which the weights are
considered) have order $p$.
Now the following result is clear.

\begin{thm}
Let $G$ be a GGS-group.
Then the maps from $\Stab_G(1)$ to $\F_p$ sending every $g\in\Stab_G(1)$ to
its partial weight with respect to one of the $b_i$ or to its total weight are
well-defined homomorphisms.
\end{thm}

\begin{thm}
\label{G'}
Let $G$ be a GGS-group.
Then the derived subgroup $G'$ consists of all the elements of $\Stab_G(1)$ whose total
weight is equal to $0$.
\end{thm}

\begin{proof}
The map $\vartheta$ sending each element of $\Stab_G(1)$ to its total weight is a
homomorphism onto the abelian group $\F_p$, and consequently $G'\le\ker\vartheta$.
Since $|G:G'|=p^2$ and $|G:\Stab_G(1)|=|\Stab_G(1):\ker\vartheta|=p$, the equality
follows.
\end{proof}

\begin{dfn}
Let $G$ be a GGS-group.
If $g\in\Stab_G(1)$ has partial weight $r_i$ with respect to $b_i$ for $i=0,\ldots,p-1$,
we say that $(r_0,\ldots,r_{p-1})\in\F_p^p$ is the \emph{weight vector\/} of $g$.
\end{dfn}

As we next see, we can analyze the subgroups $\Stab_G(2)$ and $\Stab_G(3)$ by
using the weight vector.

\begin{thm}
\label{Stab(2,3)}
Let $G$ be a GGS-group with defining vector $\mathbf{e}$, and put $C=C(\mathbf{e},0)$.
If the weight vector of $g\in\Stab_G(1)$ is $(r_0,\ldots,r_{p-1})$, then:
\begin{enumerate}
\item
We have $g\in\Stab_G(2)$ if and only if $(r_0\ \ldots \ r_{p-1})C=(0\ \ldots \ 0)$.
\item
If $g\in\Stab_G(3)$ then $(r_0,\ldots,r_{p-1})=(0,\ldots,0)$.
\end{enumerate}
\end{thm}

\begin{proof}
(i)
If we write $\psi(g)$ as in (\ref{components of word}), then $g\in\Stab_G(2)$ if
and only if $m_i=0$ in $\F_p$ for every $i=1,\ldots,p$.
Now, by (\ref{exponent}), this is equivalent to the condition
$(r_0\ \ldots \ r_{p-1})C=(0\ \ldots \ 0)$.

(ii)
Again we use the expression in (\ref{components of word}).
If $g\in\Stab_G(3)$ then $\omega_i(b_0,\ldots,b_{p-1})$\break
$\in\Stab_G(2)$ for all $i=1,\ldots,p$.
As mentioned above, $\omega_i(b_0,\ldots,b_{p-1})$ is an element of total weight $r_i$.
Let $(s_0,\ldots,s_{p-1})$ be the weight vector of this element, so that $r_i=s_0+\cdots+s_{p-1}$.
Then, by (i), we have $(s_0\ \ldots \ s_{p-1})C=(0\ \ldots \ 0)$.
Since $\rk C=\rk (C\mid \mathbf{1})$ by Lemma \ref{rank circulant}, it follows that
$r_i=s_0+\cdots+s_{p-1}=0$, as desired.
\end{proof}

One may wonder whether the converse holds in (ii) of the previous theorem,
i.e.\ if the weight vector of an element is $(0,\ldots,0)$, does it lie
in $\Stab_G(3)$?
We make things clearer in the following theorem.

\begin{thm}
\label{Stab(1)'}
Let $G$ be a GGS-group.
Then $\Stab_G(1)'$ consists of all elements of $\Stab_G(1)$ whose weight vector is
$(0,\ldots,0)$.
Furthermore, we have $|G:\Stab_G(1)'|=p^{p+1}$.
\end{thm}

\begin{proof}
The map $\rho$ which sends every element of $\Stab_G(1)$ to its weight vector is a
homomorphism onto $\F_p^p$.
Thus $|\Stab_G(1):\ker\rho|=p^p$.
Since $\F_p^p$ is abelian, it follows that $\Stab_G(1)'\le\ker\rho$.
On the other hand, since $\Stab_G(1)=\langle b_0,\ldots,b_{p-1} \rangle$ and every
$b_i$ has order $p$, we have $|\Stab_G(1):\Stab_G(1)'|\le p^p$.
Hence $\ker\rho=\Stab_G(1)'$ and $|\Stab_G(1):\Stab_G(1)'|=p^p$.
Since $|G:\Stab_G(1)|=p$, we are done.
\end{proof}

In particular, we have $\Stab_G(3)\le \Stab_G(1)'$.
Once we prove Theorem A, it will follow that $|G:\Stab_G(3)|=p^{tp+1-\delta}$,
where $t$ is the rank of $C(\mathbf{e},0)$ and $\delta$ is $1$ or $0$, according
as $\mathbf{e}$ is symmetric or not.
Since $t$ is always at least $2$, we have $|G:\Stab_G(3)|>p^{p+1}$ in every case.
Hence $\Stab_G(3)$ is always a proper subgroup of $\Stab_G(1)'$, and the converse
of (ii) in Theorem \ref{Stab(2,3)} never holds.

\vspace{10pt}

Next we prove a result which will allow us to reduce, for the calculation of
the order of congruence quotients and of the Hausdorff dimension, to the case
of GGS-groups with defining vectors of the form $\mathbf{e}=(1,e_2,\ldots,e_{p-1})$.
We need the following lemma.

\begin{lem}
\label{permutation}
Let $p$ be a prime, and let $\sigma=(1\ 2\ \ldots \ p)$.
Assume that $\alpha\in S_p$ satisfies the following two conditions:
\begin{enumerate}
\item
$\alpha$ normalizes the subgroup $\langle \sigma \rangle$.
\item
$\alpha(p)=p$.
\end{enumerate}
Then, for every $i=1,\ldots,p-1$, if $\alpha(i)=j$ we have
$\alpha(p-i)=p-j$.
\end{lem}

\begin{proof}
If we think of $S_p$ as the set of permutations of the field $\F_p$,
then $\sigma$ corresponds to the map $\ell\mapsto \ell+1$, and the
normalizer of $\langle \sigma \rangle$ in $S_p$ corresponds to the affine
group over $\F_p$ (see Lemma 14.1.2 of \cite{cox}).
Thus $\alpha(\ell)=a\ell+b$ for some $a\in\F_p^{\times}$ and $b\in\F_p$.
Since $\alpha(p)=p$, it follows that $b=0$, and so $\alpha(\ell)=a\ell$
for every $\ell\in\F_p$.
Hence $\alpha$ is a linear map and, as a consequence,
\[
\alpha(p-i) = \alpha(-i) = -\alpha(i) = -j = p-j.
\]
\end{proof}

We say that an automorphism $f$ of $\T$ has \emph{constant portrait\/}
if $f$ has the same label at all vertices of $\T$.
By formula (\ref{label laws}) for the labels of a composition,
the set of all automorphisms of constant portrait is a subgroup of $\Aut\T$.

\begin{thm}
\label{reduction}
Let $G$ be a GGS-group with defining vector $\mathbf{e}=(e_1,\ldots,e_{p-1})$,
and assume that $e_k\ne 0$.
Then there exists $f\in\Aut\T$ of constant portrait such that
$L=G^f$ is a GGS-group whose defining vector
$\mathbf{e}'=(e_1',\ldots,e_{p-1}')$ satisfies:
\begin{enumerate}
\item
$\mathbf{e}'$ is a permutation of the vector $\mathbf{e}/e_k$, that is,
there exists $\alpha\in S_{p-1}$ such that $e_i'=e_{\alpha(i)}/e_k$ for
all $i=1,\ldots,p-1$.
\item
$\alpha(1)=k$, and so $e_1'=1$.
\item
If $\alpha(i)=j$ then $\alpha(p-i)=p-j$.
In other words, two values which are placed in symmetric positions of
$\mathbf{e}$ are moved (after division by $e_k$) to symmetric positions
of $\mathbf{e}'$.
Thus $\mathbf{e}'$ is symmetric if and only if $\mathbf{e}$ is.
\item
$\rk C(\mathbf{e},0)=\rk C(\mathbf{e}',0)$.
\end{enumerate}
Furthermore, we have $|G_n|=|L_n|$ for every $n$, and
$\dim_{\Gamma} \overline G=\dim_{\Gamma} \overline L$.
\end{thm}

\begin{proof}
Observe that there exists a permutation $\beta\in S_p$, in fact only one, that normalizes
the subgroup $\langle \sigma \rangle$ and such that $\beta(k)=1$ and $\beta(p)=p$.
Indeed, since $\sigma^{\beta}=(\beta(1) \ \ldots \beta(p))$ and
the positions of $1$ and $p$ are already fixed in this last tuple, there is only one way
to choose the rest of the images of $\beta$ if we want to obtain a power of $\sigma$.
Let $r$ be defined by the condition $\sigma^{\beta}=\sigma^r$, and set $\alpha=\beta^{-1}$.
Note that $\alpha(1)=k$ and that, by Lemma \ref{permutation}, if $\alpha(i)=j$ then
$\alpha(p-i)=p-j$.

Now we define an automorphism $f$ of $\T$ by choosing the labels at all vertices of $\T$
equal to $\beta$.
We claim that $L=G^f$ satisfies the properties of the statement of the theorem.
We have
\[
(g^f)_{(v)}=\beta^{-1} g_{(f^{-1}(v))} \beta
\]
for every $g\in G$ and every vertex $v$ of the tree.
It readily follows that $a^f=a^r$.
We now consider $c=b^f$.
Let $S$ be the set of all vertices of the form $p\overset{n}{\ldots}pi$,
where $n\ge 0$ and $1\le i\le p-1$.
If $v\in S$, then we have
$f(v)=p\overset{n}{\ldots}p\beta(i)$, and consequently
$f^{-1}(v)=p\overset{n}{\ldots}p\alpha(i)$.
Thus
\[
c_{(v)}
=
\beta^{-1} b_{(p\ldots p\alpha(i))} \beta
=
(\sigma^{e_{\alpha(i)}})^{\beta}
=
\sigma^{re_{\alpha(i)}}
\]
in this case.
On the other hand, if $v\not\in S$, then also $f^{-1}(v)\not\in S$, and
so we have $b_{(f^{-1}(v))}=1$ and $c_{(v)}=1$.
Thus $c$ is the automorphism given by the recursive relation
\[
\psi(c) = (a^{re_{\alpha(1)}},\ldots,a^{re_{\alpha(p-1)}},c).
\]
Now, let $\ell$ be the inverse of $re_{\alpha(1)}$ modulo $p$, and put $b'=c^{\ell}$.
Then $L=\langle a,b' \rangle$, where $b'$ is the automorphism defined by
\[
\psi(b') = (a^{e'_1},\ldots,a^{e'_{p-1}},b'),
\]
i.e.\ $L$ is the GGS-group with defining vector $\mathbf{e}'$.
This proves (i), (ii), and (iii).

Let us now check (iv).
If $C=C(\mathbf{e},0)$, $C'=C(\mathbf{e}',0)$ and we define $e_p=0$, then
\[
c'_{ij}
=
e_{\alpha(j-i+1)}/e_k
=
e_{\alpha(j)-\alpha(i)+\alpha(1)}/e_k
=
c_{\alpha(i)-\alpha(1)+1,\alpha(j)}/e_k,
\]
since we know that $\alpha$ is a homomorphism by the proof of Lemma \ref{permutation}.
(Here, all indices are taken modulo $p$ between $1$ and $p$.)
By observing that the maps $i\mapsto \alpha(i)-\alpha(1)+1$ and $j\mapsto \alpha(j)$
are permutations of $\F_p$, we conclude that $\rk C=\rk C'$.

Finally, note that, since $G$ and $L$ are conjugate, we clearly have $|G_n|=|L_n|$,
and then by (\ref{dim-formula}), also $\dim_{\Gamma} \overline G=\dim_{\Gamma} \overline L$.
\end{proof}

We want to stress the fact that the automorphism $f$ conjugating $G$ to $L$ in
the previous theorem has constant portrait.
This has nice consequences, such as the following one.

\begin{pro}
\label{conjugating}
Let $J$ and $K$ be two subgroups of $\Aut\T$, where $J$ is contained in $\Stab(1)$.
If $f\in\Aut\T$ has constant portrait, then we have
\[
K\times \overset{p}{\cdots} \times K\subseteq \psi(J)
\]
if and only if
\[
K^f\times \overset{p}{\cdots} \times K^f\subseteq \psi(J^f).
\]
\end{pro}

\begin{proof}
Since $f^{-1}$ is also an automorphism of constant portrait, it
suffices to prove the `only if' part.
Let $\beta$ be the permutation appearing at all labels of $f$.
Then we can write $f=ch$, where $c$ is the rooted automorphism
corresponding to $\beta$ and $h\in\Stab(1)$ is such that
$\psi(h)=(f,\ldots,f)$.

Let us now consider an arbitrary tuple $(k_1,\ldots,k_p)$,
with $k_i\in K$ for every $i=1,\ldots,p$.
By hypothesis, there exists $j\in J$ such that $\psi(j)=(k_1,\ldots,k_p)$.
Then $\psi(j^c)=(k_{\beta^{-1}(1)},\ldots,k_{\beta^{-1}(p)})$, and
consequently
\[
\psi(j^f) = \psi(j^c)^{\psi(h)}
= (k_{\beta^{-1}(1)},\ldots,k_{\beta^{-1}(p)})^{(f,\ldots,f)}
= (k_{\beta^{-1}(1)}^f,\ldots,k_{\beta^{-1}(p)}^f).
\]
Clearly, this implies that $K^f\times \cdots \times K^f \subseteq \psi(J^f)$.
\end{proof}

The previous proposition will be useful when we want to find a branch structure
in a GGS-group.
The same can be said about the following result.

\begin{pro}
\label{normal closure}
Let $G$ be a GGS-group, and let $L$ and $N$ be two normal subgroups of $G$.
If $L=\langle X \rangle^G$ for a subset $X$ of $G$, and
$(x,1,\ldots,1)\in \psi(N)$ for every $x\in X$, then
\[
L\times \overset{p}{\cdots} \times L \subseteq \psi(N).
\]
\end{pro}

\begin{proof}
By Proposition \ref{surjective}, if $g\in G$ there exists $h\in \Stab_G(1)$
such that the first component of $\psi(h)$ is $g$.
Since $(x,1,\ldots,1)\in\psi(N)$ and $N$ is normal in $G$, it follows that
$(x^g,1,\ldots,1)\in \psi(N)$ for every $x\in X$ and $g\in G$.
Hence
\[
L\times \{1\} \times \overset{p-1}{\cdots} \times \{1\} \subseteq \psi(N),
\]
since $L=\langle x^g\mid x\in X,\ g\in G\rangle$.

Now, if $\psi(n)=(\ell_1,\ell_2,\ldots,\ell_p)$ then
$\psi(n^a)=(\ell_p,\ell_1,\ldots,\ell_{p-1})$.
As a consequence,
\[
\{1\} \times \cdots \times \{1\} \times L \times \{1\} \times \cdots \times \{1\}
\subseteq
\psi(N),
\]
where $L$ may appear at any position.
The result follows.
\end{proof}

\section{GGS-groups with non-constant defining vector}
\label{sec: non-constant}

In this section we prove Theorems A and B in the case that the defining vector
$\mathbf{e}$ of the GGS-group $G$ is not constant.
As it turns out, the key is to prove that $G$ has a certain branch structure.
We begin by recalling the concepts that we will need about branching in
$\Aut\T$.

\begin{dfn}
Let $G$ be a self-similar spherically transitive group of automorphisms of a
regular tree, and let $K$ be a non-trivial subgroup of $\Stab_G(1)$.
We say that $G$ is \emph{weakly regular branch\/} over $K$ if
\[
K\times \cdots \times K\subseteq \psi(K).
\]
If furthermore $K$ has finite index in $G$, we say that $G$ is
\emph{regular branch\/} over $K$.
\end{dfn}

It is well-known (and an immediate consequence of Proposition \ref{surjective})
that every GGS-group $G$ is self-similar and spherically transitive.
We next see that, if $\mathbf{e}$ is not constant, then $G$ is regular
branch over $\gamma_3(G)$.

\begin{lem}
\label{gamma3}
Let $G$ be a GGS-group with non-constant defining vector.
Then
\[
\psi(\gamma_3(\Stab_G(1)))
=
\gamma_3(G)\times \overset{p}{\cdots} \times \gamma_3(G).
\]
In particular,
\[
\gamma_3(G)\times \overset{p}{\cdots} \times \gamma_3(G)
\subseteq \psi(\gamma_3(G)),
\]
and $G$ is a regular branch group over $\gamma_3(G)$.
\end{lem}

\begin{proof}
Since $\psi(\Stab_G(1))$ is contained in $G\times \overset{p}{\cdots} \times G$,
it clearly suffices to prove the inclusion $\supseteq$.
By Theorem \ref{reduction} and Proposition \ref{conjugating}, we may assume
that $\mathbf{e}=(1,e_2,\ldots,e_{p-1})$.
If $e_{p-1}=0$ then
\[
\psi(b)
=
(a,\ldots,a^{e_{p-2}},1,b),
\]
and consequently
\[
\psi([b_0,b_1,b_0])
=
([a,b,a],1,\ldots,1)
\]
and
\[
\psi([b_0,b_1,b_1])
=
([a,b,b],1,\ldots,1).
\]
Since $G=\langle a,b \rangle$, it follows that
$\gamma_3(G)=\langle [a,b,a], [a,b,b] \rangle^G$, and then by
Proposition \ref{normal closure}, we have
$\gamma_3(G)\times \cdots \times \gamma_3(G)\subseteq \psi(\gamma_3(\Stab_G(1)))$.
Thus we may assume that $e_{p-1}\ne 0$.

Now we consider the following two cases separately:
\begin{enumerate}
\item
There exists $k\in\{2,\ldots,p-2\}$ such that $(e_{k-1},e_k)$ and
$(e_k,e_{k+1})$ are not proportional.
\item
$(e_{k-1},e_k)$ and $(e_k,e_{k+1})$ are proportional for all
$k=2,\ldots,p-2$.
\end{enumerate}
Observe that if $p=3$ then case (ii) vacuously holds.

(i)
Let us put
\[
g_k = b_{p-k+1}^{e_k}b_{p-k}^{-e_{k-1}}
\]
for $2\le k\le p-2$, so that
\[
\psi(g_k) = (a^{e_k^2-e_{k-1}e_{k+1}},\ldots,1).
\]
(The intermediate values represented by the dots are not necessarily $1$
in this case.)
Since $(e_{k-1},e_k)$ and $(e_k,e_{k+1})$ are not proportional,
we have $e_k^2-e_{k-1}e_{k+1}\ne 0$.
Hence there is a power $g$ of $g_k$ such that
\[
\psi(g) = (a,\ldots,1).
\]
On the other hand, since
\[
\psi(b_1b_{p-1}^{-e_{p-1}})
=
(ba^{-e_2e_{p-1}},\ldots,1),
\]
with the help of $g$ we can get an element $h\in\Stab_G(1)$ such that
\[
\psi(h) = (b,\ldots,1).
\]
Consequently,
\[
\psi([b_0,b_1,g])
=
([a,b,a],1,\ldots,1)
\]
and
\[
\psi([b_0,b_1,h])
=
([a,b,b],1,\ldots,1),
\]
and the result follows as before from Proposition \ref{normal closure}.

(ii)
Since $e_1=1$, it follows that $e_i=e_2^{i-1}$ for every $i=1,\ldots,p-1$.
(Note that this is valid all the same if $p=3$.)
Hence $\mathbf{e}=(1,m,m^2,\ldots,m^{p-2})$ with $m\ne 1$, because $\mathbf{e}$
is not constant.
Since $e_{p-1}\ne 0$, we also have $m\ne 0$, and consequently $m^{p-1}=1$.
Then
\[
\psi(b_0b_1^{-m}) = (ab^{-m},1,\ldots,1,ba^{-1})
\]
and
\[
\psi(b_1b_2^{-m}) = (ba^{-1},ab^{-m},1,\ldots,1).
\]
Hence
\[
\psi([b_0,b_1,b_1b_2^{-m}])
=
([a,b,ba^{-1}],1,\ldots,1)
\]
and
\[
\psi([b_2^m,b_1,b_0b_1^{-m}])
=
([a,b,ab^{-m}],1,\ldots,1).
\]
Now, since $G'=\langle [a,b] \rangle^G$ and
$\langle ab^{-m},ba^{-1} \rangle=\langle b^{1-m},ba^{-1} \rangle$ is
the whole of $G$ (at this point, it is essential that $m\ne 1$),
it follows that
\[
\gamma_3(G)=\langle [a,b,ab^{-m}], [a,b,ba^{-1}] \rangle^G.
\]
Thus the result is again a consequence of
Proposition \ref{normal closure}.
\end{proof}

As a consequence of the previous lemma, we can show that, for
$\mathbf{e}$ non-constant and $n\ge 3$, there is a close relation between
$\Stab_G(n)$ and $\Stab_G(n-1)$ in a GGS-group $G$.

\begin{lem}
\label{Stab n eta n-1}
Let $G$ be a GGS-group with non-constant defining vector $\mathbf{e}$.
Then, for every $n\ge 3$ we have
\[
\psi(\Stab_G(n))
=
\Stab_G(n-1) \times \overset{p}{\cdots} \times \Stab_G(n-1)
\]
and
\[
\psi_{n+1}(\Stab_{G_{n+1}}(n))
=
\Stab_{G_n}(n-1) \times \overset{p}{\cdots} \times \Stab_{G_n}(n-1).
\]
\end{lem}

\begin{proof}
Clearly, it suffices to prove the first equality.
By using Corollary \ref{Stab(2)} and Lemma \ref{gamma3}, we have
\[
\Stab_G(2)\times \overset{p}{\cdots} \times \Stab_G(2)
\subseteq
\gamma_3(G)\times \overset{p}{\cdots} \times \gamma_3(G)
=
\psi(\gamma_3(\Stab_G(1))).
\]
Thus $\Stab_G(n-1)\times \cdots \times \Stab_G(n-1)$ is contained in
the image of $\Stab_G(1)$ under $\psi$ for all $n\ge 3$, and the result
follows.
\end{proof}

If the vector $\mathbf{e}$ is non-symmetric, we can improve Lemma \ref{gamma3}
as follows.

\begin{lem}
\label{regular branch}
Let $G$ be a GGS-group with non-symmetric defining vector.
Then
\[
\psi(\Stab_G(1)') = G'\times \overset{p}{\cdots} \times G'.
\]
In particular,
\[
G'\times \overset{p}{\cdots} \times G' \subseteq \psi(G'),
\]
and $G$ is a regular branch group over $G'$.
\end{lem}

\begin{proof}
Observe that we only need to care about the inclusion $\supseteq$.
By Theorem \ref{reduction} and Proposition \ref{conjugating}, we may assume
that $e_1=1$ and $e_{p-1}\ne 1$, since $\mathbf{e}$ is non-symmetric.
Let us write $m$ for $e_{p-1}$.

By using (\ref{tuplak}), we get
\begin{align}
\label{tochazo}
\notag
\begin{split}
\psi([b_0,b_1])
&=
([a,b],1,\ldots,1,[b,a^m])
\\
&\equiv
([a,b],1,\ldots,1,[a,b]^{-m})
\pmod{\gamma_3(G)\times \overset{p}{\cdots} \times \gamma_3(G)},
\end{split}
\\
\notag
\begin{split}
\psi([b_{p-1},b_0]^m)
&=
(1,\ldots,1,[b,a^m]^m,[a,b]^m)
\\
&\equiv
(1,\ldots,1,[a,b]^{-m^2},[a,b]^m)
\pmod{\gamma_3(G)\times \overset{p}{\cdots} \times \gamma_3(G)},
\end{split}
\\
\notag
&\ \vdots
\\
\notag
\begin{split}
\psi([b_1,b_2]^{m^{p-1}})
&=
([b,a^m]^{m^{p-1}},[a,b]^{m^{p-1}},1,\ldots,1)
\\
&\equiv
([a,b]^{-m^p},[a,b]^{m^{p-1}},1,\ldots,1)
\pmod{\gamma_3(G)\times \overset{p}{\cdots} \times \gamma_3(G)}.
\end{split}
\end{align}
Since $m^p=m$ (recall that $m\in\F_p$), if we multiply together all the
expressions above, we obtain that
\begin{multline*}
\psi([b_0,b_1] [b_{p-1},b_0]^m \ldots [b_1,b_2]^{m^{p-1}})
\equiv
([a,b]^{1-m},1,\ldots,1)
\\
\pmod{\gamma_3(G)\times \overset{p}{\cdots} \times \gamma_3(G)}.
\end{multline*}
If we use the inclusion
\[
\gamma_3(G)\times \overset{p}{\cdots} \times \gamma_3(G)
\subseteq
\psi(\Stab_G(1)'),
\]
which is a consequence of Lemma \ref{gamma3}, we get
\[
([a,b]^{1-m},1,\ldots,1)\in \psi(\Stab_G(1)').
\]
Now, since $G=\langle a,b \rangle$ and $m\ne 1$, it follows that $G'$ is
the normal closure of $[a,b]^{1-m}$.
By Proposition \ref{normal closure}, we conclude that
$G'\times \cdots \times G'\subseteq \psi(\Stab_G(1)')$.
\end{proof}

Now we can proceed to calculate the order of $G_n$ for every $n\ge 1$,
and as a consequence, to obtain the Hausdorff dimension of $\overline G$
in $\Gamma$, provided that the defining vector $\mathbf{e}$ is not constant.
We deal separately with the following two cases: (i) $\mathbf{e}$ is
not symmetric; (ii) $\mathbf{e}$ is symmetric and not constant.
In both cases, the key is to determine the order of $\Stab_{G_3}(2)$
and to use Lemma \ref{Stab n eta n-1}.
We begin by case (i).

\begin{thm}
\label{non-symmetric stabilizer}
Let $G$ be a GGS-group with non-symmetric defining vector $\mathbf{e}$.
Then
\[
|\Stab_{G_3}(2)| = p^{t(p-1)},
\]
where $t$ is the rank of $C(\mathbf{e},0)$.
\end{thm}

\begin{proof}
By Theorem \ref{reduction}, we may assume that $e_1=1$ and
$e_{p-1}\ne 1$.
For simplicity, let us write $C$ for $C(\mathbf{e},0)$.
Since $\Stab_{G_3}(2)=\Stab_G(2)/\Stab_G(3)$, we are going to study the image of
$\Stab_G(2)$ under the canonical epimorphism $\pi$ from $G$ onto $G_3$.

Let $g$ be an arbitrary element of $\Stab_G(1)$, and let $(r_0,\ldots,r_{p-1})$
denote the weight vector of $g$.
By Theorem \ref{Stab(2,3)}, we have $g\in\Stab_G(2)$ if and only if
\[
(r_0\ r_1\ \ldots\ r_{p-1})C=(0\ 0\ \ldots\ 0).
\]
Since the rank of $C$ is $t$, this system has $p^{p-t}$ solutions,
which we denote by
\[
r^{(i)} = (r^{(i)}_0,\ldots,r^{(i)}_{p-1}),
\]
for $i=1,\ldots,p^{p-t}$.
We may assume that $r^{(1)}=(0,\ldots,0)$.

Each solution $r^{(i)}$ determines a subset $R^{(i)}$ of $\Stab_G(2)$,
consisting of all the elements whose weight vector is $r^{(i)}$.
Put $S^{(i)}=\pi(R^{(i)})$.
By the discussion in the previous paragraph, we know that $\Stab_{G_3}(2)$
is the union of all the $S^{(i)}$ for $i=1,\ldots,p^{p-t}$.
We will prove the following:
\begin{enumerate}
\item
If $i\ne j$ then $S^{(i)}$ and $S^{(j)}$ are disjoint.
(By Theorem \ref{total weight}, we know that $R^{(i)}$ and $R^{(j)}$ are
 disjoint, but we have to rule out the possibility that an element in $R^{(i)}$
 and an element in $R^{(j)}$ have the same image in $G_3$.)
\item
$|S^{(i)}|=p^{p(t-1)}$ for all $i=1,\ldots,p^{p-t}$.
\end{enumerate}
Once (i) and (ii) are proved, it readily follows that
$|\Stab_{G_3}(2)|=p^{t(p-1)}$, as desired.

We begin by proving (i).
For this purpose, assume that $g\in R^{(i)}$ and $h\in R^{(j)}$ are two elements
with the same image in $G_3$.
Then $gh^{-1}\in\Stab_G(3)$ and, by Theorem \ref{Stab(2,3)}, the weight vector of
$gh^{-1}$ is $(0,\ldots,0)$.
Since the weight vector defines a homomorphism from $\Stab_G(1)$ to $\F_p^p$,
it follows that $r^{(i)}=r^{(j)}$, and so $i=j$, as desired.

Now we proceed to the proof of (ii).
By definition, each $S^{(i)}$ is non-empty.
If $h_i$ is an element of $S^{(i)}$, then it is clear that $S^{(i)}=h_iS^{(1)}$.
Thus $|S^{(i)}|=|S^{(1)}|$, and it suffices to see that $S^{(1)}$ has the desired
cardinality.
Let $g$ be an arbitrary element of $\Stab_G(2)$.
According to (\ref{components of word}), we have $g\in R^{(1)}$ if and only if
each component of $\psi(g)$ has total weight equal to $0$.
By Theorem \ref{G'}, this is equivalent to $\psi(g)$ lying in $G'\times \cdots \times G'$.
On the other hand, since $G'\le\Stab_G(1)$, we have
$\psi^{-1}(G'\times \cdots \times G')\le \Stab(2)$.
Hence
\begin{equation}
\label{R1}
R^{(1)} = G\cap \psi^{-1}(G'\times \cdots \times G').
\end{equation}
Note that this equality is valid for any defining vector $\mathbf{e}$.
Now, since we are working under the assumption that $\mathbf{e}$ is non-symmetric,
we have $G'\times \cdots \times G'\le \psi(G')$ by Lemma \ref{regular branch}.
Thus we conclude that $R^{(1)}=\psi^{-1}(G'\times \cdots \times G')$ in this case or,
equivalently, that
\[
\psi(R^{(1)})=G'\times \cdots \times G'.
\]

We consider now the following commutative diagram:
\begin{equation}
\label{commutative diagram}
\begin{CD}
R^{(1)} @>{\pi}>> S^{(1)}
\\
@V{\psi}VV @VV{\psi_3}V
\\
G'\times \cdots \times G' @>{\tilde\pi \times \cdots \times \tilde\pi}>>
\displaystyle \frac{G'}{\Stab_G(2)} \times \cdots \times \frac{G'}{\Stab_G(2)},
\end{CD}
\end{equation}
where $\tilde\pi$ denotes reduction modulo $\Stab_G(2)$.
(Take into account that $G'$ contains $\Stab_G(2)$ by Theorem \ref{Stab(1)}.)
By the discussion of the preceding paragraph, the left vertical arrow of the
diagram is surjective.
Consequently, the right vertical arrow is also surjective, and since it is
obviously injective, it follows that it is a bijective map.
In particular,
\[
|S^{(1)}|=|G':\Stab_G(2)|^p.
\]
Now, by Theorems \ref{Stab(1)} and \ref{order of G2}, we have $|G:G'|=p^2$
and $|G:\Stab_G(2)|=p^{t+1}$.
Thus $|G':\Stab_G(2)|=p^{t-1}$, and we conclude that $|S^{(1)}|=p^{p(t-1)}$,
as desired.
\end{proof}

\begin{thm}
\label{non-symmetric}
Let $G$ be a GGS-group with non-symmetric defining vector $\mathbf{e}$.
Then
\[
\log_p |G_n| = tp^{n-2}+1,
\quad
\text{for every $n\ge 2$,}
\]
where $t$ is the rank of $C(\mathbf{e},0)$, and
\[
\dim_{\Gamma} \overline G = \frac{(p-1)t}{p^2}.
\]
\end{thm}

\begin{proof}
We argue by induction on $n\ge 2$.
By Theorem \ref{order of G2}, we have $|G_2|=p^{t+1}$.
Suppose now that $n>2$ and that the result is true for $n-1$.
By using Lemma \ref{Stab n eta n-1}, we have
\[
|\Stab_{G_n}(n-1)|=|\Stab_{G_{n-1}}(n-2)|^p=\cdots
=|\Stab_{G_3}(2)|^{p^{n-3}}.
\]
Since $|\Stab_{G_3}(2)|=p^{t(p-1)}$ by Theorem \ref{non-symmetric stabilizer},
we conclude that
\[
|G_n|=|G_{n-1}||\Stab_{G_n}(n-1)|=p^{tp^{n-3}+1} \cdot p^{tp^{n-3}(p-1)}
=p^{tp^{n-2}+1},
\]
as desired.
Finally, the value of $\dim_{\Gamma} \overline G$ follows directly from
(\ref{dim-formula}).
\end{proof}

Next we consider the case when the vector $\mathbf{e}$ is non-constant and
symmetric.

\begin{thm}
\label{symmetric stabilizer}
Let $G$ be a GGS-group with symmetric non-constant defining vector $\mathbf{e}$.
Then
\[
|\Stab_{G_3}(2)| = p^{t(p-1)-1},
\]
where $t$ is the rank of $C(\mathbf{e},0)$.
\end{thm}

\begin{proof}
Let $\pi$, $R^{(i)}$ and $S^{(i)}$ for $i=1,\ldots,p^{p-t}$ be as in
the proof of Theorem \ref{non-symmetric stabilizer}.
The plan of the proof is the same as in that theorem.
The difference is that, in this case, we need to see that
\[
|S^{(1)}|
=
p^{p(t-1)-1}.
\]
For that purpose, it suffices to prove that the image of $S^{(1)}$ under
the injective map $\psi_3$ is a subgroup of index $p$ of
\[
\frac{G'}{\Stab_G(2)}\times\cdots\times\frac{G'}{\Stab_G(2)}.
\]

We know from (\ref{R1}) that $R^{(1)}=G\cap \psi^{-1}(G'\times\cdots\times G')$
consists of all elements of $G$ whose weight vector is $(0,\ldots,0)$.
According to Theorem \ref{Stab(1)'}, we have $R^{(1)}=\Stab_G(1)'$.
Hence
\begin{equation}
\label{generators R1}
R^{(1)} = \langle [b_i,b_j]^h \mid 0\le i,j\le p-1,\ h\in\Stab_G(1) \rangle.
\end{equation}

Let us consider again the commutative diagram in (\ref{commutative diagram}).
Since
\[
\ker (\tilde\pi \times \cdots \times \tilde\pi)
=
\Stab_G(2) \times \cdots \times \Stab_G(2)
=
\psi(\Stab_G(3))
\]
by Lemma \ref{Stab n eta n-1}, and since $\Stab_G(3)\le R^{(1)}$ by
Theorem \ref{Stab(2,3)}, it follows that the index
\begin{multline*}
\left|
\frac{G'}{\Stab_G(2)}\times\cdots\times\frac{G'}{\Stab_G(2)} : \psi_3(S^{(1)})
\right|
=
\\
|(\tilde\pi \times \cdots \times \tilde\pi)(G'\times\cdots\times G')
:(\tilde\pi \times \cdots \times \tilde\pi)(\psi(R^{(1)}))|
\end{multline*}
is the same as
\[
|G'\times\cdots\times G':\psi(R^{(1)})|.
\]
Thus it suffices to prove that this last index is $p$.

Let $\overline\psi$ the map from $R^{(1)}$ to
$G'/\gamma_3(G)\times \overset{p}{\cdots} \times G'/\gamma_3(G)$ which is obtained by
first applying $\psi$ and then reducing every component modulo $\gamma_3(G)$.
Observe that $G'/\gamma_3(G)\times \overset{p}{\cdots} \times G'/\gamma_3(G)$ can be
seen as a vector space of dimension $p$ over $\F_p$, since
$|G':\gamma_3(G)|=p$.
Since we may assume that $e_1=1$, and since $e_{p-1}=e_1$, we have
\[
\psi([b_i,b_{i+1}]) = (1,\ldots,1,[b,a],[a,b],1,\ldots,1),
\quad
\text{for $i=1,\ldots,p-1$,}
\]
where $[b,a]$ appears at the $i$th position.
Now, $G'/\gamma_3(G)$ is generated by the image of $[b,a]$, and so it readily
follows that the dimension of $\overline\psi(R^{(1)})$ is at least $p-1$.
Hence
\[
|G'\times \cdots \times G': \psi(R^{(1)})(\gamma_3(G)\times \cdots \times \gamma_3(G))|
=
1 \text{ or } p.
\]
Since $\gamma_3(G)\times \cdots \times \gamma_3(G)\le \psi(R^{(1)})$ by
Lemma \ref{gamma3} and (\ref{R1}), we get
\[
|G'\times \cdots \times G':\psi(R^{(1)})|
=
1 \text{ or } p.
\]
Thus it suffices to see that $([a,b],1,\ldots,1)\not\in \psi(R^{(1)})$ in order to
conclude that $|G'\times\cdots\times G':\psi(R^{(1)})|=p$, as desired.

Let $\lambda:\Stab_G(1) \longrightarrow \F_p$ be the homomorphism given by
\[
g \longmapsto \sum_{i=0}^{p-1} \, ir_i,
\]
where $(r_0,\ldots,r_{p-1})$ is the weight vector of $g$.
If $g\in\Stab_G(1)$ then the weight vector of $g^b$ is also $(r_0,\ldots,r_{p-1})$,
and the weight vector of $g^a$ is $(r_{p-1},r_0,\ldots,r_{p-2})$.
Hence $\lambda(g^b)=\lambda(g)$, and if $g\in G'$, then furthermore
\[
\lambda(g^a)
=
\sum_{i=0}^{p-1} \, ir_{i-1}
=
\sum_{i=0}^{p-1} \, r_{i-1}
+
\sum_{i=0}^{p-1} \, (i-1)r_{i-1}
\\
=
\lambda(g),
\]
since $r_0+\cdots+r_{p-1}=0$ by Theorem \ref{G'}.
It follows that $\lambda(g^h)=\lambda(g)$ for every $g\in G'$ and $h\in G$.

Now we define $\Lambda:G'\times\cdots\times G'\longrightarrow \F_p$
by means of
\[
\Lambda(g_1,\ldots,g_p)
=
\lambda(g_1) + \cdots + \lambda(g_p).
\]
By the preceding paragraph, we have
\[
\Lambda(g^h)
=
\Lambda(g),
\quad
\text{for all $g\in G'\times\cdots\times G'$ and $h\in G\times\cdots\times G$.}
\]
Hence $\ker\Lambda$ is a normal subgroup of $G\times\cdots\times G$.

For every $1\le i<j\le p$, we have
\begin{multline*}
\psi([b_i,b_j])
=
(1,\ldots,1,[b,a^{e_{i-j}}],1,\ldots,1,[a^{e_{j-i}},b],1,\ldots,1)
=
\\
(1,\ldots,1,b_0^{-1}b_{e_{i-j}},1,\ldots,1,b_{e_{j-i}}^{-1}b_0,1,\ldots,1),
\end{multline*}
where the non-trivial components are at positions $i$ and $j$.
Since $\mathbf{e}$ is symmetric, we have $e_{i-j}=e_{j-i}$, and consequently
\[
\Lambda(\psi([b_i,b_j]))
=
e_{i-j}-e_{j-i}
=
0.
\]
Hence $\psi([b_i,b_j])\in \ker\Lambda$, and since $\ker\Lambda$ is a normal
subgroup of $G\times \cdots \times G$, it follows from  (\ref{generators R1})
that $\psi(R^{(1)})\le\ker\Lambda$.
Since
\[
\Lambda([a,b],1,\ldots,1)
=
\Lambda(b_1^{-1}b_0,1,\ldots,1)
=
-1,
\]
we deduce that $([a,b],1,\ldots,1)\not\in\psi(R^{(1)})$, which completes the proof.
\end{proof}

\begin{thm}
\label{symmetric}
Let $G$ be a GGS-group with a non-constant symmetric defining vector $\mathbf{e}$.
Then
\[
\log_p |G_n| = tp^{n-2}+1-\frac{p^{n-2}-1}{p-1},
\quad
\text{for every $n\ge 2$,}
\]
where $t$ is the rank of $C(\mathbf{e},0)$, and
\[
\dim_{\Gamma} \overline G = \frac{(p-1)t-1}{p^2}.
\]
\end{thm}

\begin{proof}
The proof is completely similar to that of Theorem \ref{non-symmetric}.
\end{proof}

\section{GGS-groups with constant defining vector}

In this section, we deal with the case where the defining vector
is constant, say $\mathbf{e}=(e,\ldots,e)$, where $e\in\F_p^{\times}$.
Let $m$ be the inverse of $e$ in $\F_p^{\times}$, and $b^*=b^m$.
Then $G=\langle a,b^* \rangle$, and $\psi(b^*)=(a,\ldots,a,b^*)$.
For this reason, we may assume in the remainder of this section
that $\mathbf{e}=(1,\ldots,1)$.

We begin by defining a sequence of elements of $G$ that will be fundamental
in the sequel.
We put $y_0=ba^{-1}$ and, more generally, $y_i=y_0^{a^i}$ for every integer $i$.
Thus $y_i^{a^j}=y_{i+j}$ for all $i,j\in\Z$.
Also,
\begin{equation}
\label{yi conjugate by b}
y_i^b = y_i^{aa^{-1}b} = y_{i+1}^{y_1}.
\end{equation}
Observe that $y_i=y_j$ if $i\equiv j\pmod p$, so that the set $\{y_0,\ldots,y_{p-1}\}$
already contains all the $y_i$.
In the following lemma, we collect some important properties of the elements $y_i$.
We adopt the following convention: given a vector $v$ of length $p$ and an integer $i$,
not lying in the range $\{1,\ldots,p\}$, the $i$th position of $v$ is to be understood
as the $j$th position, where $j\in\{1,\ldots,p\}$ and $i\equiv j\pmod p$.

\begin{lem}
\label{[yi,yj]}
Let $G$ be a GGS-group with constant defining vector.
Then:
\begin{enumerate}
\item
$y_{p-1}y_{p-2}\ldots y_1y_0=1$.
\item
If $z_i$ is the tuple of length $p$ having $y_2$ at position $i-2$,
$y_1^{-1}$ at position $i-1$, and $1$ elsewhere, then
\begin{equation}
\label{formula [yi,yj]}
\psi([y_i,y_j]) = z_i z_j^{-1},
\quad
\text{for every $i$ and $j$.}
\end{equation}
\item
We have
\begin{equation}
\label{decomposition [yi,yj]}
[y_i,y_j] = [y_i,y_{i-1}] [y_{i-1},y_{i-2}] \ldots [y_{j+1},y_j],
\quad
\text{for every $i>j$.}
\end{equation}
\end{enumerate}
\end{lem}

\begin{proof}
(i)
We have
\begin{multline*}
y_{p-1}y_{p-2}\ldots y_1y_0
=
a^{-(p-1)}ba^{p-2} \cdot a^{-(p-2)}ba^{p-3} \ldots a^{-1}b \cdot ba^{-1}
\\
=
a^{-(p-1)} b^p a^{-1}
=
1.
\end{multline*}

(ii)
Clearly, it is enough to see the result for $i>j$.
On the other hand, since both sequences $\{y_i\}$ and $\{z_i\}$ are periodic
of period $p$, we may assume that $i$ and $j$ lie in the set $\{3,\ldots,p+2\}$.
If $r=j-3$ and $k=i-r$, then
\[
[y_i,y_j] = [y_k^{a^r},y_3^{a^r}] = [y_k,y_3]^{a^r},
\]
and so $\psi([y_i,y_j])$ is the result of applying to $\psi([y_k,y_3])$
the permutation which moves every element $r$ positions to the right.
It readily follows that it suffices to prove (\ref{formula [yi,yj]}) for
$[y_k,y_3]$ with $4\le k\le p+2$.

Since $y_i=a^{-i}ba^{i-1}=a^{-1}b_{i-1}$ for every $i$, we have
\begin{equation}
\label{[yk,y3]}
[y_k,y_3]
=
b_{k-1}^{-1} a b_2^{-1} b_{k-1} a^{-1} b_2
=
b_{k-1}^{-1} b_1^{-1} b_{k-2} b_2
=
(b_1^{-1} b_{k-2})^{b_{k-1}} (b_{k-1}^{-1} b_2).
\end{equation}
Now, it follows from (\ref{tuplak}) that
\begin{multline*}
\psi((b_1^{-1} b_{k-2})^{b_{k-1}})
=
(y_1^{-1},1,\overset{k-4}{\ldots},1,y_1,1,\ldots,1)^%
{(a,\overset{k-2}{\ldots},a,b,a,\ldots,a)}
\\
=
\begin{cases}
(y_2^{-1},1,\overset{k-4}{\ldots},1,y_2,1,\ldots,1),
& \text{if $4\le k\le p+1$,}
\\
(y_1^{-1}y_2^{-1}y_1,1,\ldots,1,y_2),
& \text{if $k=p+2$.}
\end{cases}
\end{multline*}
Here, we have used that $y_1^b=y_2^{y_1}$ by (\ref{yi conjugate by b}).
Similarly,
\[
\psi(b_{k-1}^{-1} b_2)
=
\begin{cases}
(1,y_1,1,\overset{k-4}{\ldots},1,y_1^{-1},1,\ldots,1),
& \text{if $4\le k\le p+1$,}
\\
(y_1^{-1},y_1,1,\ldots,1),
& \text{if $k=p+2$.}
\end{cases}
\]
By taking these values to (\ref{[yk,y3]}), we obtain that
$\psi([y_k,y_3])=z_kz_3^{-1}$, as desired.

(iii)
This follows immediately from (ii), since
\begin{align*}
\psi([y_i,y_j])
&=
(z_iz_{i-1}^{-1}) (z_{i-1}z_{i-2}^{-1}) \ldots (z_{j+1}z_j^{-1})
\\
&=
\psi([y_i,y_{i-1}]) \psi([y_{i-1},y_{i-2}]) \ldots \psi([y_{j+1},y_j])
\\
&=
\psi([y_i,y_{i-1}] [y_{i-1},y_{i-2}] \ldots [y_{j+1},y_j]).
\end{align*}
\end{proof}

Next we introduce a maximal subgroup $K$ of $G$ that will play a
key role in the determination of the order of $G_n$ in the
case that $\mathbf{e}$ is constant.

\begin{lem}
\label{G weakly regular branch}
Let $G$ be a GGS-group with constant defining vector, and
let $K=\langle ba^{-1} \rangle^G$.
Then:
\begin{enumerate}
\item
$G'\le K$ and $|G:K|=p$.
\item
$K = \langle y_0,y_1,\ldots,y_{p-1} \rangle$ and
$K' = \langle [y_1,y_0] \rangle^G$.
\item
$K'\times \overset{p}{\cdots} \times K'
\subseteq \psi(K')
\subseteq \psi(G')
\subseteq K\times \overset{p}{\cdots} \times K$.
In particular, $G$ is a weakly regular branch group over $K'$.
\item
If $L=\psi^{-1}(K'\times \overset{p}{\cdots} \times K')$ (which, by
(iii), is contained in $K'$), then the conjugates $[y_{i+1},y_i]^{b^j}$,
where $0\le i,j\le p-1$, generate $K'$ modulo $L$.
\end{enumerate}
\end{lem}

\begin{proof}
(i)
Since $[a,ba^{-1}]=[a,b]^{a^{-1}}\in K$ and $K$ is normal in $G$, it follows
that $G'$ is contained in $K$.
Then $|G:K|=|G/G':K/G'|=p$.

(ii)
Let us first prove that $K=\langle y_0,y_1,\ldots,y_{p-1} \rangle$.
For this purpose, it suffices to see that
$N=\langle y_0,y_1,\ldots,y_{p-1} \rangle$ is a normal subgroup of $G$.
This is clear, since $y_i^a=y_{i+1}$ and $y_i^b=y_{i+1}^{y_1}$ for every $i$.

It follows that
\[
K' = \langle [y_i,y_j] \mid 0\le j<i \le p-1 \rangle^K
= \langle [y_i,y_j] \mid 0\le j<i \le p-1 \rangle^G,
\]
where the second equality holds because $K'$ is normal in $G$.
By (\ref{decomposition [yi,yj]}), every commutator $[y_i,y_j]$ with $0\le j<i \le p-1$
can be expressed in terms of the $[y_k,y_{k-1}]$ with $k=1,\ldots,p-1$.
Since $[y_k,y_{k-1}]=[y_1,y_0]^{a^{k-1}}$, we conclude that
$K'=\langle [y_1,y_0] \rangle^G$.

(iii)
Let us first prove the inclusion $\psi(G')\subseteq K\times \overset{p}{\cdots} \times K$.
We have
\begin{multline*}
\psi([b,a])
=
\psi(b^{-1}b^a)
=
(a^{-1},a^{-1},\ldots,a^{-1},b^{-1}) (b,a,\ldots,a,a)
\\
=
(a^{-1}b,1,\ldots,1,b^{-1}a)\in K\times \overset{p}{\cdots} \times K.
\end{multline*}
Now, since $K$ is normal in $G$, it readily follows that
\[
\psi([b,a]^g)\in K\times \overset{p}{\cdots} \times K,
\quad
\text{for every $g\in G$.}
\]
This proves the desired inclusion.

Now we focus on proving that
$K'\times \overset{p}{\cdots} \times K'\subseteq \psi(K')$.
By Proposition \ref{normal closure} and (ii), it suffices to see that
\[
([y_1,y_0],1,\ldots,1)\in \psi(K').
\]
We consider separately the cases $p\ge 5$ and $p=3$.

Suppose first that $p\ge 5$.
By using (\ref{formula [yi,yj]}), we have
\[
\psi([y_1,y_2]) = (y_1,1,\ldots,1,y_2,y_1^{-1}y_2^{-1})
\]
and
\[
\psi([y_3,y_4]) = (y_2,y_1^{-1}y_2^{-1},y_1,1,\ldots,1).
\]
If $k=[[y_3,y_4],[y_1,y_2]]$, it follows that
\[
\psi(k) = ([y_2,y_1],1,\ldots,1),
\]
since $p\ge 5$.
Hence
\[
([y_1,y_0],1,\ldots,1) = \psi(k^{b^{-1}}) \in \psi(K'),
\]
as desired.

Assume now that $p=3$.
We have
\[
\psi([y_1,y_0]) = (y_1y_0,y_0^{-1},y_1^{-1}),
\]
since $y_2y_1y_0=1$, by (i) of Lemma \ref{[yi,yj]}.
Hence
\begin{align*}
\psi([y_0,y_1]^b)
&=
(y_0^{-1}y_1^{-1},y_0,y_1)^{(a,a,b)}
=
(y_1^{-1}y_2^{-1},y_1,y_1^b)
\\
&=
((y_2y_1)^{-1},y_1,y_2^{y_1})
=
(y_0,y_1,(y_0^{-1}y_1^{-1})^{y_1})
\\
&=
(y_0,y_1,y_1^{-1}y_0^{-1}),
\end{align*}
and
\[
([y_1,y_0],1,1) = \psi([y_0,y_1]^{ba}[y_1,y_0]) \in \psi(K'),
\]
which completes the proof.

(iv)
Let us consider an arbitrary element $g\in G$, and let us write
$g=ha^ib^j$, for some $i,j\in\Z$, $h\in G'$.
Then
\[
[y_1,y_0]^g = ([y_1,y_0][y_1,y_0,h])^{a^ib^j}
\equiv
[y_1,y_0]^{a^ib^j} = [y_{i+1},y_i]^{b^j}
\pmod{L},
\]
since $\psi([y_1,y_0,h])\in \psi(G'')
\subseteq K'\times \overset{p}{\cdots} \times K'$ by (iii).
Now, since the conjugates $[y_1,y_0]^g$ generate $K'$ by (ii), the
result follows.
\end{proof}

In the following results, we consider the action of an element of $G$
by conjugation as an endomorphism of $K/K'$, which allows us to multiply
several conjugates of an element of $K$, modulo $K'$, by adding the
elements by which we are conjugating.
This gives a meaning to expressions like $g^{1+a+\cdots+a^{p-1}}\in K'$
for an element $g\in K$.

\begin{lem}
\label{sum of powers of a}
Let $G$ be a GGS-group with constant defining vector, and
let $K=\langle ba^{-1} \rangle^G$.
If $g\in K$ then
\[
g^{1+a+\cdots+a^{p-1}} \in K'.
\]
\end{lem}

\begin{proof}
The map $R$ sending $g\in K$ to $g^{1+a+\cdots+a^{p-1}}K'$ is a well-defined
homomorphism from $K$ to $K/K'$, and we want to see that $R$ is the
trivial homomorphism.
Since $K=\langle y_0,\ldots,y_{p-1} \rangle$ by (ii) of
Lemma \ref{G weakly regular branch}, it suffices to check that
$y_i\in\ker R$ for every $i$.
Now,
\[
R(y_i) = y_iy_{i+1} \ldots y_{p-1}y_0 \ldots y_{i-1} K'
= y_{p-1}y_{p-2}\ldots y_1y_0 K'
= K'
\]
by (i) of Lemma \ref{[yi,yj]}, and we are done.
\end{proof}

\begin{lem}
\label{conditions to be in K'}
Let $G$ be a GGS-group with constant defining vector, and
let $K=\langle ba^{-1} \rangle^G$.
If $g\in K'$ and we write $\psi(g)=(g_1,\ldots,g_p)$, then:
\begin{enumerate}
\item
$g_pg_{p-1}\ldots g_1\in K'$.
\item
$\prod_{i=1}^{p-1} \, g_i^{a+a^2+\cdots+a^i} \in K'$.
\end{enumerate}
Similarly, if $g\in K'\Stab_G(n)$ for some $n\ge 1$, then both
$g_pg_{p-1}\ldots g_1$ and $\prod_{i=1}^{p-1} \, g_i^{a+a^2+\cdots+a^i}$
lie in $K'\Stab_G(n-1)$.
\end{lem}

\begin{proof}
We first deal with the case that $g\in K'$.
Let us consider the following two maps:
\[
\begin{matrix}
P & : & K\times \overset{p}{\cdots} \times K & \longrightarrow & K/K'
\\
& & (g_1,\ldots,g_p) & \longmapsto & g_p\ldots g_1 K',
\end{matrix}
\]
and
\[
\begin{matrix}
Q & : & K\times \overset{p}{\cdots} \times K & \longrightarrow & K/K'
\\
& & (g_1,\ldots,g_p) & \longmapsto
& \prod_{i=1}^{p-1} \, g_i^{a+a^2+\cdots+a^i} K'.
\end{matrix}
\]
Clearly, $P$ and $Q$ are homomorphisms.
By (iii) of Lemma \ref{G weakly regular branch}, $\psi(K')$ is contained in
the domain of $P$ and $Q$, and our goal is to prove that it is actually
in the kernels of these maps.
Since the image of $K'\times \overset{p}{\cdots} \times K'$ is trivial,
it suffices to see that $\psi(g)\in \ker P$ and $\psi(g)\in \ker Q$ for every $g$
in a system of generators of $K'$ modulo $L$, where
$L=\psi^{-1}(K'\times \overset{p}{\cdots} \times K')$.
By (iv) of Lemma \ref{G weakly regular branch}, the conjugates
$[y_{i+1},y_i]^{b^j}$, for $i,j=0,\ldots,p-1$ constitute such a set of generators.

Let $c\in\Gamma$ be defined by means of $\psi(c)=(a,a,\ldots,a)$.
We claim that
\begin{equation}
\label{b and c}
g^b \equiv g^c \pmod{L},
\quad
\text{for every $g\in K'$.}
\end{equation}
Indeed, we have $\psi(b)=\psi(c) (1,\ldots,1,a^{-1}b)$, and so
\begin{multline*}
\psi(g^b) = \psi(g^c)^{(1,\ldots,1,a^{-1}b)}
= \psi(g^c) [\psi(g^c),(1,\ldots,1,a^{-1}b)]
\\
\equiv
\psi(g^c) \pmod{K'\times \overset{p}{\cdots} \times K'},
\end{multline*}
since $\psi(g^c)\in K\times \overset{p}{\cdots} \times K$ and $a^{-1}b\in K$.

As a consequence of (\ref{b and c}), it suffices to see that
$\psi([y_{i+1},y_i]^{c^j})$ lies in both $\ker P$ and $\ker Q$.
Since
\[
P(\psi([y_{i+1},y_i]^{c^j}))=P(\psi([y_{i+1},y_i]))^{a^j}
\]
and
\[
Q(\psi([y_{i+1},y_i]^{c^j}))=Q(\psi([y_{i+1},y_i]))^{a^j},
\]
we have reduced ourselves to proving that $\psi([y_{i+1},y_i])$ is in the
kernel of $P$ and $Q$ for every $i$.
According to (\ref{formula [yi,yj]}), we have $\psi([y_{i+1},y_i])=z_{i+1}z_i^{-1}$,
with $z_i$ as defined in Lemma \ref{[yi,yj]}.
Now, one can easily check that
\[
P(z_i)=y_1^{-1}y_2K'
\quad
\text{and}
\quad
Q(z_i)=y_2^{-1}K'
\quad
\text{for every $i$,}
\]
where in the case of $Q$ and $i=1$ we need to use that
\[
y_2^{a+a^2+\cdots+a^{p-1}}\equiv y_2^{-1} \pmod{K'},
\]
by Lemma \ref{sum of powers of a}.
It readily follows that $\psi([y_{i+1},y_i])$ lies in both $\ker P$ and $\ker Q$,
as desired.

Assume now that $g\in K'\Stab_G(n)$, and let us write $g=fh$, with $f\in K'$
and $h\in\Stab_G(n)$.
Put $\psi(f)=(f_1,\ldots,f_p)$ and $\psi(h)=(h_1,\ldots,h_p)$.
Since $h_1,\ldots,h_p\in\Stab_G(n-1)$, which is a normal subgroup of $G$,
we have
\[
g_p\ldots g_1 = f_ph_p \ldots f_1h_1 = f_p \ldots f_1 h^*,
\]
for some $h^*\in\Stab_G(n-1)$.
Since $f\in K'$, we already know that $f_p\ldots f_1\in K'$, and so
we conclude that $g_p\ldots g_1\in K'\Stab_G(n-1)$, as desired.
The second assertion can be proved in a similar way.
\end{proof}

\begin{thm}
\label{isomorphisms K}
Let $G$ be a GGS-group with constant defining vector, and
let $K=\langle ba^{-1} \rangle^G$ and
$L=\psi^{-1}(K'\times \overset{p}{\cdots} \times K')$.
Then the following isomorphisms hold:
\[
K'/L \cong K/K' \times \overset{p-2}{\cdots} \times K/K',
\]
and
\[
K'\Stab_G(n)/L\Stab_G(n)
\cong
K/K'\Stab_G(n-1) \times \overset{p-2}{\cdots} \times K/K'\Stab_G(n-1),
\]
for every $n\ge 3$.
\end{thm}

\begin{proof}
Let $\pi$ be the map given by
\[
\begin{matrix}
K\times \overset{p}{\cdots} \times K & \longrightarrow
& K/K'\times \overset{p-2}{\cdots} \times K/K'
\\
(g_1,\ldots,g_p) & \longmapsto & (g_1K',\ldots,g_{p-2}K'),
\end{matrix}
\]
and let $R$ be the composition of
$\psi:K'\longrightarrow K\times \overset{p}{\cdots} \times K$
with $\pi$.
If we see that $R$ is surjective, and that $\ker R=L$, then
the first isomorphism of the statement follows.

Let $g\in K'$ be an element lying in $\ker R$.
If $\psi(g)=(g_1,\ldots,g_p)$, then we have $g_1,\ldots,g_{p-2}\in K'$.
By (ii) of Lemma \ref{conditions to be in K'}, it follows that
\[
g_{p-1}^{a+\cdots+a^{p-1}} \in K',
\]
and by applying Lemma \ref{sum of powers of a}, we get $g_{p-1}\in K'$.
Now, (i) of Lemma \ref{conditions to be in K'} immediately yields that
also $g_p\in K'$.
This proves that $\ker R=L$.

Now we prove that
\begin{equation}
\label{surjective onto first}
K/K'\times \{ \overline 1 \} \times \cdots \times \{ \overline 1 \}
\subseteq R(K').
\end{equation}
Then, by arguing as in the proof of Proposition \ref{normal closure},
it follows that $R$ is surjective.
By (\ref{formula [yi,yj]}), we have
\[
\psi([y_1,y_2])
=
(y_1,1,\ldots,1,h_{p-1},h_p)
\]
for some elements $h_{p-1},h_p\in K$.
Hence
\[
\psi([y_1,y_2]^{b^{i-1}})
=
(y_i,1,\ldots,1,h_{p-1}^*,h_p^*)
\]
for every $i$, and we are done, since $K=\langle y_0,\ldots,y_{p-1} \rangle$.

The second isomorphism can be proved in a similar way.
Observe that the condition $n\ge 3$ guarantees that $\Stab_G(n-1)\le G'\le K$,
so that it makes sense to write $K/K'\Stab_G(n-1)$.
Consider this time the homomorphism
\[
\begin{matrix}
\pi_n & : & K\times \overset{p}{\cdots} \times K & \longrightarrow
& K/K'\Stab_G(n-1)\times \overset{p-2}{\cdots} \times K/K'\Stab_G(n-1)
\\[5pt]
& & (g_1,\ldots,g_p) & \longmapsto & (g_1K'\Stab_G(n-1),\ldots,g_{p-2}K'\Stab_G(n-1)),
\end{matrix}
\]
and let $R_n$ be the composition of
$\psi:K'\longrightarrow K\times \overset{p}{\cdots} \times K$
with $\pi_n$.
Observe that the surjectiveness of $R$ already implies that $R_n$ is surjective.
Let us prove that $\ker R_n=L\Stab_G(n)\cap K'$.
The same proof as above, but using the last part of Lemma \ref{conditions to be in K'},
shows that
\begin{multline*}
\psi(\ker R_n)
=
( K'\Stab_G(n-1) \times \overset{p}{\cdots} \times K'\Stab_G(n-1) )
\cap \psi(K')
\\
=
( K' \times \overset{p}{\cdots} \times K' )
( \Stab_G(n-1) \times \overset{p}{\cdots} \times \Stab_G(n-1) )
\cap \psi(K').
\end{multline*}
Since $K' \times \overset{p}{\cdots} \times K' \subseteq \psi(K')$,
we can apply Dedekind's Law to get
\[
\psi(\ker R_n)
=
( K' \times \overset{p}{\cdots} \times K' )
\big( (\Stab_G(n-1) \times \overset{p}{\cdots} \times \Stab_G(n-1))
\cap \psi(K') \big).
\]
Now, since $n\ge 3$, we have
\begin{multline*}
( \Stab_G(n-1) \times \overset{p}{\cdots} \times \Stab_G(n-1) )
\cap \psi(K')
=
\psi(\Stab_G(n))\cap \psi(K')
\\
=
\psi(\Stab_G(n)\cap K'),
\end{multline*}
and it follows that
\begin{multline*}
\psi(\ker R_n)
=
( K' \times \overset{p}{\cdots} \times K' )
\psi(\Stab_G(n)\cap K')
=
\psi(L)
\psi(\Stab_G(n)\cap K')
\\
=
\psi(L(\Stab_G(n)\cap K')).
\end{multline*}
Hence
\[
\ker R_n = L(\Stab_G(n)\cap K') = L\Stab_G(n)\cap K',
\]
as claimed.

Now, we can readily obtain the desired isomorphism:
\begin{multline*}
K'\Stab_G(n)/L\Stab_G(n)
\cong
K'/(L\Stab_G(n)\cap K')
=
K'/\ker R_n
\\
\cong
R_n(K')
=
K/K'\Stab_G(n-1)\times \overset{p-2}{\cdots} \times K/K'\Stab_G(n-1).
\end{multline*}
\end{proof}

\begin{thm}
\label{quotient maximal class}
Let $G$ be a GGS-group with constant defining vector, and
let $K=\langle ba^{-1} \rangle^G$.
Then, for every $n\ge 2$, the quotient $G/K'\Stab_G(n)$ is a $p$-group of
maximal class of order $p^{n+1}$.
\end{thm}

\begin{proof}
For simplicity, let us write $T_n=K'\Stab_G(n)$, $Q_n=G/T_n$ and
$A_n=K/T_n$ (take into account that $\Stab_G(2)\le G'\le K$).
Since $|Q_n:Q_n'|=|G:G'|=p^2$ and $A_n$ is an abelian maximal subgroup of $Q_n$,
it follows from Lemma \ref{maximal class} that $Q_n$ is a $p$-group
of maximal class.
As a consequence, if we want to prove that $|Q_n|=p^{n+1}$, it suffices to
see that the nilpotency class of $Q_n$ is $n$.

We need an auxiliary result.
Let $\{x_i\}_{i\ge 1}$ be a sequence of elements of $G$ such that
$\{x_1,x_2\}=\{a,b\}$ and $x_i\in\{a,b\}$ for every $i\ge 3$.
We claim that, for every $i\ge 2$, the section
$\gamma_i(Q_n)/\gamma_{i+1}(Q_n)$ is generated by the image of the
commutator $[x_1,x_2,\ldots,x_i]$.
We argue by induction on $i$.
If $i=2$ then we have to show that the image of $[a,b]$ generates
$\gamma_2(Q_n)/\gamma_3(Q_n)$.
This follows immediately from (i) in Lemma \ref{maximal class}, since
$[a,b]=[a,a^{-1}b]$, where $bT_n\in Q_n\smallsetminus A_n$ and
$a^{-1}bT_n=(ba^{-1}T_n)^a\in A_n\smallsetminus \gamma_2(Q_n)$.
Now, if we assume that the result holds for $i-1$, we get it for $i$
by using (ii) of Lemma \ref{maximal class}.

Let us now prove that the class of $Q_n$ is $n$, by induction on $n$.
Assume first that $n=2$.
We have
\[
\psi([b,a]) = (a^{-1}b,1,\ldots,1,b^{-1}a)
\]
and
\[
\psi([b,a,b]) = ([a^{-1}b,a],1,\ldots,1,[b^{-1}a,b])
= ([b,a],1,\ldots,1,[a,b]),
\]
so that $[b,a,b]\in\Stab_G(2)$.
It follows that the image of $[b,a,b]$ in $Q_2$ is trivial.
By the previous paragraph, we necessarily have $\gamma_3(Q_2)=\gamma_4(Q_2)$.
Hence $\gamma_3(Q_2)=1$, and the class of $Q_2$ is at most $2$.
If $Q_2$ is of class $1$, then $[b,a]\in K'\Stab_G(2)$ and, by
Lemma \ref{conditions to be in K'}, $a^{-1}b\in K'\Stab_G(1)$.
Hence $a^{-1}\in\Stab_G(1)$, which is a contradiction.
Thus $Q_2$ is of class $2$.

Now we assume the result for $n-1$, and we prove it for $n$.
We have
\[
\psi([b,a,b,\overset{n-1}{\ldots},b])
=
([b,a,\overset{n-1}{\ldots},a],1,\ldots,1,[a,b,\overset{n-1}{\ldots},b]),
\]
and
\[
[b,a,\overset{n-1}{\ldots},a],[a,b,\overset{n-1}{\ldots},b]
\in
K'\Stab_G(n-1),
\]
since $Q_{n-1}$ has class $n-1$ by the induction hypothesis.
Thus
\begin{equation}
\label{psi of long commutator}
\psi([b,a,b,\overset{n-1}{\ldots},b])
\in
K'\Stab_G(n-1)\times \overset{p}{\cdots} \times K'\Stab_G(n-1).
\end{equation}
Now,
\begin{align*}
(K'
\Stab_G&(n-1)
\times \overset{p}{\cdots} \times K'\Stab_G(n-1))
\cap \psi(G)
\\
&=
(K'\times \overset{p}{\cdots} \times K')
(\Stab_G(n-1)\times \overset{p}{\cdots} \times \Stab_G(n-1))
\cap \psi(G)
\\
&\subseteq
\psi(K')(\Stab_G(n-1)\times \overset{p}{\cdots} \times \Stab_G(n-1))
\cap \psi(G)
\\
&=
\psi(K')(\Stab_G(n-1)\times \overset{p}{\cdots} \times \Stab_G(n-1)
\cap \psi(G))
\\
&=
\psi(K')\psi(\Stab_G(n))
=
\psi(K'\Stab_G(n)).
\end{align*}
It follows that $[b,a,b,\overset{n-1}{\ldots},b]\in K'\Stab_G(n)$,
and so this commutator becomes trivial in $Q_n$.
Since the image of this commutator generates the quotient
$\gamma_{n+1}(Q_n)/\gamma_{n+2}(Q_n)$, we have $\gamma_{n+1}(Q_n)=1$.
Hence the class of $Q_n$ is at most $n$.

If $Q_n$ has class strictly less than $n$, then since the image of
$[b,a,b,\overset{n-2}{\ldots},b]$ generates $\gamma_n(Q_n)/\gamma_{n+1}(Q_n)$,
it follows that
\[
[b,a,b,\overset{n-2}{\ldots},b]\in K'\Stab_G(n).
\]
Since
\[
\psi([b,a,b,\overset{n-2}{\ldots},b])
=
([b,a,\overset{n-2}{\ldots},a],1,\ldots,1,[a,b,\overset{n-2}{\ldots},b]),
\]
it follows from Lemma \ref{conditions to be in K'} that
\[
[b,a,\overset{n-2}{\ldots},a]\in K'\Stab_G(n-1).
\]
This is a contradiction, since $Q_{n-1}$ is of class $n-1$, and
$\gamma_{n-1}(Q_{n-1})/\gamma_n(Q_{n-1})$ is generated by the image of
$[b,a,\overset{n-2}{\ldots},a]$.
Thus we conclude that the nilpotency class of $Q_n$ is $n$, which
completes the proof of the theorem.
\end{proof}

\begin{thm}
\label{symmetric}
Let $G$ be a GGS-group with a constant defining vector.
Then
\[
\log_p |G_n| = p^{n-1}+1-\frac{p^{n-2}-1}{p-1}-\frac{p^{n-2}-(n-2)p+n-3}{(p-1)^2},
\]
for every $n\ge 2$, and
\[
\dim_{\Gamma} \overline G = \frac{p-2}{p-1}.
\]
\end{thm}

\begin{proof}
As on previous occasions, the formula for the Hausdorff dimension of
$\overline G$ is immediate once we obtain $\log_p |G_n|$.
For that purpose, we argue by induction on $n$.
If $n=2$, then by Theorem \ref{order of G2}, we have $\log_p |G_2|=t+1$,
where $t$ is the rank of the matrix $C=C(1,\overset{p-1}{\ldots},1,0)$.
By Lemma \ref{rank circulant}, $p-t$ is the multiplicity of $1$ as a root
in $\F_p$ of the polynomial $X^{p-2}+\cdots+X+1$.
Thus $t=p$ and $\log_p |G_2|=p+1$, as desired.

Assume now that $n\ge 3$.
Let $K=\langle ba^{-1} \rangle^G$, and
$L=\psi^{-1}(K'\times \overset{p}{\cdots} \times K')$.
Then we have the following decomposition of the order of $G_n$:
\begin{equation}
\label{decomposition |Gn|}
|G_n|
=
|G:K'\Stab_G(n)||K'\Stab_G(n):L\Stab_G(n)|
|L\Stab_G(n):\Stab_G(n)|.
\end{equation}
By Theorem \ref{quotient maximal class}, we know that
$|G:K'\Stab_G(n)|=p^{n+1}$.
On the other hand, since
\[
K'\Stab_G(n)/L\Stab_G(n)
\cong
K/K'\Stab_G(n-1) \times \overset{p-2}{\cdots} \times K/K'\Stab_G(n-1)
\]
by Theorem \ref{isomorphisms K}, and since $|K/K'\Stab_G(n-1)|=p^{n-1}$
(again by Theorem \ref{quotient maximal class}), it follows that
\[
|K'\Stab_G(n):L\Stab_G(n)|
=
p^{(n-1)(p-2)}.
\]
Finally,
\begin{align*}
|L\Stab_G(n)
&:\Stab_G(n)|
=
|L:\Stab_L(n)|
=
|\psi(L):\psi(\Stab_L(n))|
\\
&=
|K'\times \overset{p}{\cdots} \times K':
\Stab_{K'}(n-1)\times \overset{p}{\cdots} \times \Stab_{K'}(n-1)|
\\
&=
|K':\Stab_{K'}(n-1)|^p
=
|K'\Stab_G(n-1):\Stab_G(n-1)|^p
\\
&=
|G/\Stab_G(n-1)|^p / |G/K'\Stab_G(n-1)|^p
\\
&=
|G_{n-1}|^p p^{-np}.
\end{align*}
Now, from (\ref{decomposition |Gn|}) we get
\begin{multline*}
\log_p |G_n| = p\log_p |G_{n-1}| +n+1+(n-1)(p-2)-np
\\
=
p\log_p |G_{n-1}| -n-p+3,
\end{multline*}
and the result follows by applying the induction hypothesis to $G_{n-1}$.
\end{proof}

\end{document}